\documentclass[12pt]{article}
\usepackage{amsmath}

\usepackage{mathrsfs}
\usepackage{amscd}
\usepackage{amscd, amssymb, amsmath, amsthm, graphics}
\usepackage{amsmath,amsfonts,amsthm,amssymb}
\usepackage{latexsym,amsmath}
\usepackage{graphicx,psfrag}
\usepackage[all]{xy}
\usepackage{latexsym}
\usepackage{enumerate}
\usepackage{verbatim}
\usepackage{pb-diagram}

\newtheorem{theorem}{Theorem}[section]
\newtheorem{lemma}[theorem]{Lemma}
\newtheorem{prop}[theorem]{Proposition}
\newtheorem{cor}[theorem]{Corollary}

\newtheorem{conj}[theorem]{Conjecture}

\theoremstyle{definition}
\newtheorem{definition}[theorem]{Definition}

\theoremstyle{remark}
\newtheorem{remark}[theorem]{Remark}

\newtheorem*{claim}{Claim}

\numberwithin{equation}{section}

\newcommand{\CC}{\mathbb{C}}
\newcommand{\RR}{\mathbb{R}}
\newcommand{\QQ}{\mathbb{Q}}
\newcommand{\ZZ}{\mathbb{Z}}

\newcommand{\JJ}{\mathcal{J}}
\newcommand{\CCC}{\mathcal{C}}
\newcommand{\EE}{\mathcal{E}}
\newcommand{\KK}{\mathcal{K}}
\newcommand{\LL}{\mathcal{L}}
\newcommand{\fF}{\mathcal{F}}

\newcommand{\conjone}{\begin{conj}}
\newcommand{\conjtwo}{\end{conj}}
\newcommand{\thmone}{\begin{theorem}}
\newcommand{\thmtwo}{\end{theorem}}
\newcommand{\lemmaone}{\begin{lemma}}
\newcommand{\lemmatwo}{\end{lemma}}
\newcommand{\pfone}{\begin{proof}}
\newcommand{\pftwo}{\end{proof}}
\newcommand{\defone}{\begin{definition}}
\newcommand{\deftwo}{\end{definition}}
\newcommand{\corone}{\begin{cor}}
\newcommand{\cortwo}{\end{cor}}
\newcommand{\cone}{\begin{claim}}
\newcommand{\ctwo}{\end{claim}}
\newcommand{\propone}{\begin{prop}}
\newcommand{\proptwo}{\end{prop}}
\newcommand{\eqone}{\begin{equation}}
\newcommand{\eqtwo}{\end{equation}}
\newcommand{\rmkone}{\begin{remark}}
\newcommand{\rmktwo}{\end{remark}}
\newcommand{\enone}{\begin{enumerate}}
\newcommand{\entwo}{\end{enumerate}}
\newcommand{\itone}{\begin{itemize}}
\newcommand{\ittwo}{\end{itemize}}

\newcommand{\onehalf}{\left(\begin{array}{cc}}
\newcommand{\theother}{\end{array}\right)}

\newcommand{\gijk}{\Gamma_{ijk}}
\newcommand{\gone}{\Gamma_{123}}

\newcommand{\oneeq}{\begin{equation}}
\newcommand{\twoeq}{\end{equation}}
\newcommand{\nono}{\noindent}

\begin{document}

\title{Lagrangian spheres, symplectic surfaces and the symplectic mapping class group}         % Enter your title between curly braces
\author{Tian-Jun Li, Weiwei Wu}

\maketitle

\abstract{\nono Given a Lagrangian sphere   in a symplectic
$4$-manifold $(M, \omega)$ with $b^+=1$, we find  embedded
symplectic surfaces intersecting it minimally. When the Kodaira
dimension $\kappa$ of $(M, \omega)$ is $-\infty$, this minimal intersection property turns out to be
very powerful for both the uniqueness and existence problems of Lagrangian
spheres. On the uniqueness side, for a symplectic rational manifold and any class which is not
characteristic and ternary,
we show that homologous Lagrangian spheres are smoothly isotopic, and when
the  Euler number is less than 8,
%we show that the space of homologous Lagrangian spheres is connected, generalizing
we generalize Hind and Evans' Hamiltonian uniqueness  in the monotone case.
%we show that the Torelli part of $Symp(M, \omega)$ acts transitively
%on homologous Lagrangian spheres;
% if the rational manifold has Euler number less
%than 8, we show $Ham(M, \omega)$ already acts transitively.
 On the existence side, when $\kappa=-\infty$, we give a characterization of
 classes represented by Lagrangian spheres, which  enables us to
%combining some results of homological action of diffeomorphism,
   describe the non-Torelli part of the symplectic mapping class group.}

\tableofcontents
\section{Introduction}

For a symplectic 4-manifold $(M, \omega)$, symplectic surfaces and Lagrangian surfaces are  of complementary dimensions.
Thus we can ask what can be said about their intersection pattern.
 Welschinger  investigated this problem for a Lagrangian
torus $L$ in \cite{Wel}, where he proves that
 the class $[L]$  pairs trivially with any
effective  class,
%represented by a $J-$holomorphic curve for any $\omega-$compatible almost complex structure $J$,
and  a symplectic sphere
 with positive Chern number can be isotoped symplectically away from $L$.

  In the case when $L$ is  a Lagrangian sphere in $S^2\times S^2$ with a product symplectic form, Hind \cite{Hind} constructed
  two transverse foliations of symplectic spheres where each sphere intersects $L$ in a single point.
  This is used to show that every such $L$ is Hamiltonian isotopic to the antidiagonal.
  For a Lagrangian  sphere $L$ in   a  symplectic Del
Pezzo surface with Euler number at most $7$,  Evans showed in \cite{Evans}
that it can be displaced from certain symplectic spheres with
positive Chern number up to
 Hamiltonian isotopy, and applied this displacement result to prove the uniqueness of  Hamiltonian isotopy class of Lagrangian spheres.

In section \ref{minimal-intersection}, we generalize Evans' displacement result in two ways, the first being

\thmone\label{main} Let $L$ be a Lagrangian sphere in a
 symplectic 4-manifold $(M,\omega)$, and
$A\in H_2(M;\ZZ)$ with $A^2\geq -1$.  Suppose $A$ is represented by a  symplectic sphere $C$.
Then $C$ can be isotoped symplectically to another representative of $A$ which intersects $L$ minimally.
 \thmtwo

In this paper all surfaces are smooth, embedded, connected, and oriented.
 We say that two closed surfaces intersect \textit{minimally}
if they intersect transversely at $|k|$ points where $k$ is the
homological intersection number.

The second generalization is  for symplectic surfaces of arbitrary
genus in manifolds with $b^+=1$.
 To state it let  $\EE_\omega$  be  the
set of $\omega-$exceptional classes:
$$\{E\in H_2(M, \ZZ):E \text{ is represented by an $\omega$-symplectic } (-1) \text{ sphere}\}.$$

\thmone\label{general} Suppose $(M,\omega)$ is a symplectic
4-manifold with $b^+=1$ and  $L$ is a Lagrangian
sphere.
 Assume $A\in H_2(M, \ZZ)$ satisfies $\omega(A)>0, A^2>0$ and  $A\cdot E\geq 0$
 for all $E\in\EE_\omega$.  Then there exists a symplectic
surface in the class $nA$ intersecting $L$ minimally for large
$n\in\mathbb{N}$.\thmtwo

%It is known that a $4$-manifold $(M, \omega)$ with $b^+=1$ has abundant symplectic surfaces, in  particular,
%large multiples of any square positive
%class are represented by smooth embedded connected symplectic
%surfaces up to a change of orientation. If $(M, \omega)$ is not minimal, this is still true if
%the class is further required to pair non-negatively with any $\omega-$exceptional class.

%Our proof of the theorems is an adaption of that of \cite{Evans} and
%\cite{Wel}.  Note that most of the technical-looking assumptions are indeed
%the (almost) minimal assumptions to ensure the classes $A$ or $nA$ have non-vanishing Gromov-Taubes invariant.

% We also apply Theorem \ref{main} to present an
%alternative proof of Hind's fundamental uniqueness theorems.

One consequence of
%Theorem \ref{main} and
Theorem \ref{general} is that we are able to effectively perform  the Lagrangian-relative
inflation procedure when $b^+=1$ (Section \ref{section:Lagrangian spherical}).
%and its implications.

This turns out useful in dealing with a variety of questions, especially  the existence of Lagrangian spheres.
%In view of these results
%and Seidel's example, it seems natural to ask the following question,
%of which the answer is very likely to be positive, and will appear in upcoming work:
%\begin{question} Suppose $(M, \omega)$ is a symplectic rational manifold with Euler number bigger than $7$.
%Is it true that the Hamiltonian subgroup is always a proper subgroup
%of $Symp_h(M, \omega)$,
%  i.e. the Torelli part of the symplectic mapping class group is non-trivial?
%\end{question}
To approach this question, it is convenient to introduce the following definition.

  \defone\label{def:K-Lag spherical}
A class $\xi$ is called
\textit{$K_{\omega}$-null spherical} if $\xi^2=-2, K_{\omega}(\xi)=0$ and it is represented by a
smooth sphere.  Here $K_{\omega}$ is the symplectic canonical class.
\deftwo
%Lagrangian sphere with respect to some symplectic form with $K$ as its canonical class.
%Notice that a similar
%classification for exceptional spheres appeared in \cite{MSE}.

We classify $K$-null spherical classes in any $(M,\omega)$ with $\kappa=-\infty$.
Recall that $\kappa(M, \omega)$ is the \textit{Kodaira dimension} of $(M, \omega)$ (see for example \cite{Kod}).
$\kappa$ takes values in the set $\{-\infty, 0, 1, 2\}$, and
$\kappa(M, \omega)=-\infty$ exactly when $(M, \omega)$ is symplectic rational or ruled.
The classification of $K_{\omega}$-null spherical classes, together with  the Lagrangian-relative inflation,
 enables us  to further show that the obvious necessary
condition for the existence of a Lagrangian sphere  in $(M, \omega)$ is also sufficient.

\thmone\label{Lagrangian sphere class classification} Let
$(M,\omega)$ be a symplectic 4-manifold with $\kappa=-\infty$.  $\xi\in H_2(M;\ZZ)$ is represented by a Lagrangian
sphere if and only if $\xi$ is $K_{\omega}$-null spherical and
$\omega(\xi)=0$.
\thmtwo%\label{existence of Lagrangian}

On the other hand, as in \cite{Evans}, Theorem \ref{main} is useful in establishing uniqueness results for
\textit{rational manifolds}.  A rational
manifold is  $\CC P^2\# k\overline{\CC P}^2$ or $S^2\times S^2$.  When $M$ is a rational manifold $(M,\omega)$
 is called a \textit{symplectic rational manifold}.    A
symplectic rational manifold $(M,\omega)$ which is monotone, i.e.
$[\omega]=K_{\omega}$, is also called a \textit{symplectic Del Pezzo
surface}.
%Theorems \ref{main} and \ref{general} assert that certain symplectic surfaces can be ``push-off''
%in symplectic rational manifolds. This turned out useful in the existence and uniqueness of Lagrangian spheres.
%For example, Theorem \ref{main} can be applied to establish the following Hamiltonian uniqueness of
%Lagrangian spheres in small symplectic rational manifolds:

\thmone\label{Lagrangian isotopy uniqueness}
Let $(M,\omega)$ be a  symplectic rational manifold  with Euler number $\chi\leq7$, and $\xi$  a $K_{\omega}$-null spherical class with $\omega(\xi)=0$.
If $\xi$ is not characteristic when $\chi=6$, then
 Lagrangian spheres in $\xi$ are unique up to
 Hamiltonian isotopy.

\thmtwo

This was due to Hind (\cite{Hind}) in the case of $S^2\times S^2$, and to
Evans  (\cite{Evans}) for  symplectic Del Pezzo surfaces with Euler
number up to $7$.  Notice that this is equivalent to the transitivity of the Hamiltonian
group action on the space of homologous
Lagrangian spheres.
The proof of Theorem \ref{Lagrangian isotopy
uniqueness} will be presented in Section \ref{section: Lagrangian uniqueness}. We believe that the uniqueness still holds when $\chi=6$ and $\xi$ is characteristic.
However, the condition $\chi\leq7$ in Theorem \ref{Lagrangian isotopy uniqueness} is necessary, demonstrated by Seidel's
 twisted Lagrangian spheres in symplectic Del Pezzo surfaces with  $\chi\geq 8$ (\cite{Seidel notes}).

 Further, we prove:
 \thmone\label{smooth isotopy uniqueness}
Let $(M,\omega)$ be a  symplectic rational manifold, and $\xi$  a $K_{\omega}$-null spherical class with $\omega(\xi)=0$.
If $\xi$ is not characteristic when $\chi=6$, then
 Lagrangian spheres in $\xi$ are unique up to
 smooth isotopy.

\thmtwo
 In the monotone case this was again due to Evans  (\cite{EvansT}).
 We expect the extra condition being non-characteristic when $\chi=6$ will eventually be removed.
 In fact, we are not aware of examples of  homologous but not smoothly isotopic Lagrangian spheres in any symplectic 4-manifolds. For Lagrangian tori,  such examples in a primitive homology class  were first constructed by Vidussi in \cite{Vi}, and null-homologous ones were further constructed 
by Fintushel and Stern in \cite{FS}.

We also conjecture the following version of uniqueness.

 \conjone\label{k-blowups} For any two homologous Lagrangian spheres
$L_1$ and $L_2$ in a symplectic rational manifold $(M,\omega)$, there
exists $\phi\in Symp_h(M,\omega)$ such that $\phi(L_1)=L_2$.
\conjtwo

 In other words,  the Torelli part $Symp_h(M,
\omega)$, which is the subgroup of $Symp(M, \omega)$ acting trivially on
homology,  should also act transitively on the space of Lagrangian spheres in a
fixed homology class. Evans \cite{EvansS} calculated explicitly the
homotopy type of $Symp_h(M,\omega)$ when $(M, \omega)$ is a
symplectic Del Pezzo surface with $\chi\leq 8$ (also known to
M.Pinnsonault). In particular,  when $\chi\leq 7$, it is
connected thus agreeing with  $Ham(M,\omega)$.  In our upcoming work
\cite{LW} we will
 extend the connectedness to the non-monotone case.

It turns out that we are able to  calculate the non-Torelli part of the symplectic mapping class group from  Theorem \ref{Lagrangian sphere class classification}.
Recall that each Lagrangian sphere $L$ gives rise to a symplectomorphism, well defined up to isotopy (see \cite{Seidel notes}  and 2.1.1), which
is denoted by $\tau_L$ and called the Lagrangian Dehn twist along $L$.

\thmone\label{homological action theorem} Let $(M,\omega)$ be a
symplectic 4-manifold with $\kappa=-\infty$. Then the homological action
of $Symp(M, \omega)$ is generated by Lagrangian Dehn
twists. In other words, for any  $f\in Symp(M, \omega)$, there are Lagrangian spheres $L_i$ such that
$f_*=(\tau_{L_1})_*\circ(\tau_{L_2})_*\circ\cdots\circ(\tau_{L_r})_*$.
 \thmtwo

%This theorem
%asserts that   the non-Torelli part of the symplectic
%mapping class group is detected by Lagrangian spheres.
In the homological level, Theorem \ref{homological action theorem}
 could  be viewed as a symplectic version of a classical theorem
of M. Noether, which asserts that a birational automorphism of $\CC
P^2$ (also known as \textit{plane Cremona map}) can be decomposed
into a series of \textit{ordinary quadratic transformations} (see
\cite{AC} for a complete account).\\

\nono\textit{Acknowledgement:}  The authors would like to thank
Richard Hind for his interest in our work and innumerable inspiring
 comments, as well as pointing out an error in an earlier draft.
 We would also like to thank Robert Gompf, Jonathan Evans, Chris Wendl,
 Ke Zhu, Weiyi Zhang and Chung-I Ho for helpful conversations.
After the paper was completed, we received a manuscript by
V.V.Shevchishin \cite{Shev}, where he also proved Theorems 1.4 and 1.8
using a different approach.

\section{SFT of Lagrangian $S^2$}\label{sft}

\subsection{Geometry of $T^*S^2$}

We first recall some standard facts of $T^*S^2$. Consider the embedding of the unit sphere in $\RR^3$,
which induces a symplectic embedding of $T^*S^2$ into $T^*\RR^3=\RR^3\times\RR^3$.
 In terms of the coordinates
$(u, v)\in\RR^3\times\RR^3$, $T^*S^2$ is thus given by equations (\cite{Seidel notes}, \cite{Evans}):
\begin{equation}\label{coordinate}
\{(u,v)\in\RR^3\times \RR^3:|u|=1, u\cdot v=0)\},
\end{equation}
and  the symplectic form is the restriction of $\omega_{can}=d\lambda_{can}=\sum dv_jdu_j$ on $\RR^6$,
where  the Liouville form $\lambda_{can}=\sum v_j du_j$ is also well-defined.
\eqref{coordinate}  provides a Lagrangian splitting of the tangent bundle of $T^*S^2$
into the horizontal $u-$direction and the vertical $v-$direction.

Here is another useful model. Consider the affine quadric $Q=\{z_1^2+z_2^2+z_3^2=1\}\subset \CC^3$.
In terms of $u=\hbox{Re }z \in \RR^3$ and $v=\hbox{Im }z\in \RR^3$, $Q$ is described by $|u|^2-|v|^2=1,
u\cdot v=0$. Therefore $(u, v)\to (-\frac{u}{|u|}, v|u|)$ is a diffeomorphism from $Q$ to $T^*S^2$.
Moreover, if we restrict $\omega_{can}$ on $\RR^6$ to $Q$, the diffeomorphism is in fact a symplectomorphism.

\subsubsection{Symplectomorphisms  of $T^*S^2$}

The symplectomorphism  group of $T^*S^2$ contains some compact subgroups.
For each $l>0$, denote
$T^*_l S^2$ to be the open
disk bundle  with  $|v|<l$, and $H_l$ the sphere bundle of length $l$.
The isometry group of $S^2$, $SO(3)$, acts on $(T^*S^2, \omega_{can})$ as symplectomorphisms preserving each $H_l$.

The Hamiltonian function $Z(u, v)={1\over 2} |v|^2$ generates a circle action on $T^*S^2$, agreeing with the cogeodesic flow.
If we apply the symplectic cut operation in \cite{Lerman} to  $\overline{T_{l}^*S^2}$ along $H_l$, we obtain $S^2\times S^2$  with a monotone symplectic form  (see for example \cite{Audin}).
In other words,
$T_1^*S^2$ embeds into a monotone $S^2\times S^2$  as the complement of the diagonal
$\Delta$.

 The mapping class group of the
compactly supported symplectomorphism group of $(T^*S^2, \omega_{can})$ is non-trivial. In fact, it is the infinite cyclic group generated by a model Dehn twist of the zero section (\cite{Seidel}).

To define the model  Dehn twist, consider the  Hamiltonian function $T(u, v)=|v|$ on  $T^*S^2\backslash \{\text{zero section}\}$, whose Hamiltonian vector field is
the unit field $(v/|v|, 0)$. The induced circle action is
$$
\sigma_t(u,v)=(  \cos(t)u+\sin(t){v\over |v|}, \cos(t)v-\sin(t)|v|u).
$$
Notice that $\sigma_{\pi}$ is the antipodal map $A(u,v)=(-u, -v)$, which extends smoothly over the zero section.
Now choose a function $\rho:\mathbb R\to \mathbb R$ satisfying $\rho(t)=0$ for $t\gg 0$ and $\rho(-t)=\rho(t)-t$.
The Hamiltonian flow of $\rho(T)$ is $\sigma_{t\rho'(|v|)}(u,v)$. Since $\rho'(0)=1/2$, the time $2\pi$ map extends
smoothly over the zero section as the antipodal map. The resulting compactly supported symplectomorphism  $\tau(u, v)$ of $T^*S^2$
is called a model Dehn twist.

There is a smooth isotopy with compact support from $\tau^2$ to the identity, but no such symplectic isotopies exist.

\subsubsection{Contact  geometry of  sphere bundles}

The length $l$ sphere bundle  $H_l=\{|v|=l\}$  is a contact manifold with  contact form $\lambda_{can}$. At the point $(u, v)$ the
contact plane distribution $\xi=\ker \lambda_{can}$  is spanned by $(u\times v, 0)$ and $(0, u\times v)$.

The Reeb vector field at $(u, v)$ is the  vector field
$(v, 0)$. Thus there are two dimensional simple Reeb orbits, all with the same period, and they foliate $H_l$.
This is a special case of a Reeb flow of Morse-Bott type.
 In particular,
%on the unit sphere bundle,
 the Reeb
flow agrees with the cogeodesic flow of $S^2$ with round metric.

The vector fields $(u\times v, 0)$ and $(0, u\times v)$ provide a global trivialization $\Phi$ of $\xi$. With respect to  $\Phi$,
the action of the Reeb flow on $\xi$ along any Reeb orbit in $H_l$ is considered as a path of matrices in $sp(2, \mathbb R)$, whose Maslov
index is defined to be
the  Conley-Zehnder index of the orbit (\cite{CoZ}, \cite{SZ}).
From the calculation in
\cite{Hind} (see also \cite{Evans}),
  simple Reeb orbits have Conley-Zehnder index $2$.

$H_l$ is in fact a contact-type hypersurface in $T^*_{l+\epsilon}S^2$, where  the Liouville vector field
%on $T^*S^2\backslash\{\text{zero section}\}$
is $(0, v)$.
In particular,  $\overline{T^*_{l}S^2}=\{|v|\leq l\}$ is a
Liouville domain with convex boundary $H_l$.

\subsubsection{Cylindrical coordinates}
To apply SFT, we need to change to cylindrical coordinates.
Consider a diffeomorphism  $\Psi:T^*S^2\to T^*S^2$,  $(u, v)\to (u, \psi(|v|)v/|v|)$, where    $\psi:[0, \infty)\to [0, \infty)$ is a  smooth increasing function
such that
$\psi(s)=s$ for $s$ small, and $\phi(s)=e^s$ for $s> r$.
%Here $r$ satisfies $r<l<2r$.
$\Psi$ is the identity near the zero section,
and  $(T^*S^2, \Psi^*\omega_{can})$ is a symplectic manifold with one positive  cylindrical end. Let $\omega=\Psi^*\omega_{can}$.

Then
$(T_l^*S^2, \omega)$ is still a Liouville domain, with the Liouville field given by the unit field $\eta=(0, v/|v|)$ for $|v|> r$.
Moreover,  $(T^*S^2, \omega)$ is
the (cylindrical) symplectic completion  of  $(T_l^*S^2, \omega)$.

On $H_l$, the contact form   is $\lambda_l={\psi(l)\over l}\lambda$, and
the Reeb vector field at $(u, v)$ is $R_l=({l\over \psi(l)}v, 0)$.

\subsection{Lagrangian $S^2$ and good almost complex structures}\label{section:Lagrangian and good a.c.s.}
Let $L\subset (M, \omega)$ be a Lagrangian two sphere.   From the
Weinstein neighborhood theorem, the Lagrangian sphere $L$ has a
neighborhood $U$ symplectomorphic to $(T^*_{2r} S^2, \omega_{can})$ for some small $r>0$. %Here $S^2$
%is equipped with the round metric and $U=T^*_{2R}S^2$ is the open
%set of cotangent vectors with norm $<2R$ in $T^*S^2$.
Denote the symplectomorphism by $\Xi$. Let
$U_l=\Xi^{-1}(\overline{T^*_{l} S^2})$ for $l<2r$, and
$W_l=M\backslash U_l$ be the complement of $U_l$.

In
particular, $H=\partial U_l$ is a contact-type hypersurface with contact form $\lambda=\Xi_*^{-1} \lambda_l$.
\subsubsection{$J_{t}^0$ on $T^*S^2$}
Following \cite{Hind},
we make a specific
choice of $\omega-$compatible almost complex structure $J^0$ on $T^*S^2$ as follows:  near the zero section,
 $J^0(X, 0)=(0, X)$; and for $|v|>r$,
$$J^0|_{(u,v)}(v, 0)=(0, {\psi(l)\over l}v), \quad J^0|_{(u,v)}(u\times v, 0)=(0, u\times v).$$
$J^0$ is $SO(3)$-invariant, and $J^0$ is adjusted in the sense that, for $|v|>r$,  it is ${\partial\over \partial s}-$invariant, sending the Liouville field  to the Reeb field.

Choose $l\in (r, 2r)$. When restricted to the Liouville domain $(\overline{T^*_lS^2}, \omega)$, $J^0|_{\overline{T^*_lS^2}}$ is adjusted in the collar neighborhood $r<|v|\leq l$,
and its cylindrical completion is canonically identified with   $(T^*S^2, J^0)$.
%\begin{equation}
%J(a, m)(h, k)=-(-\alpha(m)(k),J'(m)\pi k+jX(m))
%\end{equation}

We need to further consider a deformation $J^0_t$ of  $J^0$.
Let $V_t=[-t-\epsilon,t+\epsilon]$ and  $\beta_t:V_t\rightarrow[-\epsilon,\epsilon]$ be a
strictly increasing function with $\beta_t(s)=s+t$ on
$[-t-\epsilon,-t-\epsilon/2]$ and $\beta_t(s)=s-t$ on
$[t+\epsilon/2,t+\epsilon]$.
Define a smooth embedding
$f_t:V_t\times H_l\rightarrow T^*S^2$ by:
$$f_t(s,m)=(\beta_t(s)+l, m).$$
Let $\bar J_t$ be the ${\partial\over \partial s}-$invariant almost complex structure on $V_t\times H_l$
such that $\bar J_t({\partial\over \partial s})=R_l$ and $\bar J_t|_{\xi}=J^0|_{\xi}$.
Glue the almost complex manifold $(T^*S^2\backslash f_t(V_t\times H_l), J^0)$ to $(V_t\times H_l, \bar J_t)$
via $f_t$ to obtain the family of almost complex structures $J^0_t$ on
 $T^*S^2$.

 Notice that  each $J^0_t$ agrees with $J^0$ away from the collar $l-\epsilon<|v|<l+\epsilon$. And on this collar,  it agrees with $J^0$
on $\xi$, while $J^0_t|_{(u,v)}(v, 0)=(0, {d\beta_t^{-1}\over ds}|_{s= |v|-l}{\psi(l)\over l}v)$.

On the other hand,  via $f_t$,  $J^0_t$ restricted to  $\overline{T^*_lS^2}$ is the same as $J^0$ on $\overline{T^*_{l+t}S^2}$.
In particular, $J^0_{\infty}$ can be viewed an almost complex structure on $T^*S^2$, which is in fact equal to $J^0$.

%%%%%%%%%%%%%%%%%%%%%%%%%%%%% down

%%%%%%%%%%%%%%%%%%%%%%%%%%%%%%%%%%%%%%%%% up

\subsubsection{Neck-stretching  on $M$}\label{section:good almost complex
structures}

We say that an almost complex structure
 $J$ on $M$ is \textit{adjusted} to
$H=\partial U_l$ with respect to the Liouville vector field
$\Xi^{-1}_*(\eta)$, if in a tubular neighborhood of $H$, $J$ is
invariant under the flow $\Xi^{-1}_*(\eta)$,
$J(\Xi^{-1}_*(\eta))$ is the Reeb vector field on $H$, and $J$ preserves the contact plane field   $\zeta$
 defined by the contact structure $i_\eta\omega$.

Following \cite{EvansT} consider the following Fr\'echet manifold of adjusted almost complex structures:

\begin{equation}\label{J-condition}
\overline{\JJ}=\{J\in \JJ_\omega:J=\Xi^{-1}_*J^0\text{ on }U\}.
 \end{equation}

%\nono  We call an  almost complex structure a neck-stretching datum.
\nono Given $J\in \overline{\JJ}$, define
 $$ J_t=J \text{ on } X\backslash U, \quad  J_t=\Xi^{-1}_*J_t^0\text{ on }U.$$
 Notice that $J_t$ is in fact the
\textit{neck-stretching} of the adjusted $J$ along
$\partial U_l$ with respect to $\Xi^{-1}_*(\eta)$.
Fix a sequence $\{t_i\in\RR:
t_i\rightarrow+\infty\}$, we further define a sequence of Fr\'echet manifolds of adjusted almost complex structures:
$$\overline{\JJ}(i)=\{J\in \JJ_\omega: J=\Xi^{-1}_*J^0_{t_i} \text{ in } U\}.$$
\nono From the explicit description of $J_{t_i}$ in 2.2.1, we can reverse the
neck-stretching, thus there is a  diffeomorphism
$P_i:\overline{\JJ}(i)\rightarrow \overline{\JJ}$.

 When $i\to \infty$ the neck-stretching process  results in an almost complex structure $J_{\infty}$ on
 the union of symplectic
completions $\overline{W}$ and $\overline{U}$ of $W$ and $U_r$.
$\overline{W}$ and $\overline{U}$ are two open symplectic
manifolds with cylindrical ends, with $(\overline{U}, J^{\infty})$ being  $(T^*S^2, J^0)$.
$J_{\infty}$ on  the
cylindrical end of $\overline{W}$ can be described explicitly: one simply extends $\eta$ in the
obvious way, and endows an $\eta$-adjusted almost complex structure
which still restricts to $J$ on $\zeta$ as above.

To describe the  limits of
pseudo-holomorphic curves under the deformation $J_t$, we need another open symplectic manifold.
%manifold which comes from rescaling the degeneration of the almost
%complex structure near $H$.
Let $SH$ be the symplectization of the
contact manifold $H$.  We endow $SH$ again the $\eta$-adjusted
almost complex structure as on the cylindrical ends of
$\overline{W}$ and $\overline{U}$, and also denote it by $J_{\infty}$.

%$\overline{W}$ has a negative end, $\overline{U}$ has a positive end,
%$SH$ has two ends, one positive and one negative.

\subsection{Finite energy holomorphic curves}

Suppose $S$ is a closed Riemann surface
and $\Gamma\subset S$ an ordered  finite set of
punctures.

Let $(Z,  \omega)$ be any of the three symplectic 4-manifolds $\overline{W}$, $\overline{U}$, or
$SH$, each equipped with the adjusted almost complex structure $J_{\infty}$.
Denote $E^+$ ($E^-$) to be the positive (negative) end, which is allowed to be empty.

  Notice that, since $J^0_{\infty}({\partial\over \partial s})=R_l$, and $\xi$ is $J^0_{\infty}-$invariant,  the real trivialization $\Phi$ of $\xi$ on $H_l$  canonically induces a complex trivialization of the complex rank 2 bundle
$(TZ, J_{\infty})$ along $E^{\pm}$, which we still denote by $\Phi$.

Suppose $u:S\backslash \Gamma\to Z$ is a proper map.
$u$ is called simple if it does not factor through a multiple cover.

Let $u^{\pm}$ be the restriction to $u^{-1}(E^{\pm})$.
Then $u^{\pm}$ has the form $(a_{\pm}, v_{\pm})$ in coordinates $\RR_{\pm} \times H$. Consider the set $\mathcal C$ of
 functions  $\phi_{\pm}:\RR_{\pm}\to \RR$ with
 integral $1$.

The $\lambda-$energy of a  map $u:S\backslash \Gamma\to Z$ is defined by
$$E_{\lambda}(u)=\sup _{\phi_{\pm}\in \mathcal C}(\int_{u^{-1}(E^{+})}(\phi_{+}\circ a_{+})da_{+}\wedge v_{+}^*\lambda+
\int_{u^{-1}(E^{-})}(\phi_{-}\circ a_{-})da_{-}\wedge v_{-}^*\lambda).
$$
The energy of $u$ is then given by
$$E(u)=\int_{u^{-1}(Z\backslash(E^+\cup E^-))} u^*\omega+ E_{\lambda}(u).$$
$u$ is called a finite energy map if $E(u)< \infty$.
Since we are in the Morse-Bott situation, i.e the Reeb flow on $E^{\pm}$ is Morse-Bott, finite energy $J_{\infty}-$holomorphic curves
are asymptotic to periodic orbits in $E^{\pm}$ (\cite{B}).

Suppose $S$ has genus $g$, and $u$ has $s^+$ positive punctures converging to $\gamma^+_k, 1\leq k\leq s^+$, $s^-$ negative punctures
converging to $\gamma^-_k, 1\leq k\leq s^-$.
Two such maps $u$ and $u'$ are called equivalent if there is a biholomorphism $h:(S, \Gamma)\to (S', \Gamma')$ such that $u=u'\circ h$.

Each $u$ is associated with a CR operator, and $u$ is called (SFT) regular if the operator is surjective (\cite{Evans}). Denote the index of this operator by
$\text{index}(u)$. To state the index formula,
suppose $n_i^+=\hbox{cov}(\gamma_i^+)$ and $n_j^-=\hbox{cov}(\gamma_j^-)$, where  $\hbox{cov}(\gamma)$ denotes the multiplicity of $\gamma$ over a simple Reeb orbit.  Since each Reeb orbit is in a 2 dimensional manifold and has CZ index 2,
following the computation on \cite{Hind} and \cite{B},  we  have:

\eqone\label{general index formula}
 \text{index}(u)=-(2-2g)+2s+2c_1^\Phi([u])+\sum^{s^+}_{k=1}2\hbox{cov} (\gamma^+_k)-\sum^{s^-}_{k=1}2\hbox{cov}(\gamma^-_k).\eqtwo
Here  $c_1^{\Phi}(TZ)$
is the relative first Chern class  of $(TZ, J_{\infty})$  relative to the trivialization $\Phi$ along the ends,
$[u]$ is the relative homology class of $u$ (\cite{Evans}).
The following is a very special case of a theorem due to Wendl, which states that for certain $u$, the SFT regularity is automatic.

\thmone[Wendl, \cite{Wendl}] \label{at} Suppose $(W,J)$ is a $4$-dimensional
almost complex manifold with cylindrical end modelled on contact manifolds foliated by Morse-Bott Reeb orbits, and
$u:(S, \Gamma)\to W$ is a embedded pseudo-holomorphic curve with punctures.
If
\eqone\label{Wendl's equation} {\rm{index}}(u)>2g+2|\Gamma|-2,\eqtwo

\nono then $u$ is regular.
\thmtwo

\subsubsection{Regular holomorphic curves in $\overline W$}
We discuss the SFT transversality in $\overline W$.

\rmkone\label{open set} It is well-known, for example by Remark 3.2.3 in \cite{MSJ} that,
to achieve  transversality for the moduli space of pseudo-holomorphic curves,
it suffices to consider the space of $\omega-$compatible almost complex structure which
is fixed on an open set, provided that every pseudo-holomorphic curve representing the
class passes through its complement.
\rmktwo

Recall that a Baire set is the countable intersection of open and dense sets.
Since no punctured pseudo-holomorphic curves can lie completely inside $\overline U$, the arguments to prove Theorem 5.22 in \cite{Evans}
also proves:
\propone\label{evans}
Using notations in Section \ref{section:Lagrangian and good a.c.s.}, there exists a Baire set in $\bar \JJ_W\subset \bar \JJ$ such that for any $J\in \bar \JJ_W$,
$J_{\infty}$ is  SFT regular in the sense that every finite energy $J_{\infty}-$holomorphic curve $u$ is regular.
% for any data $(g, s^+, s^-, \{n_i^+\}, \{n_j^-\})$.
\proptwo

 We will need  variations of other standard transversality results about pseudo-holomorphic curves, where the above  observation will be crucial.

\subsubsection{Genus $0$  curves in $SH$ with a single simple asymptote}

In $SH$ we will encounter curves as in the following lemma.

\lemmaone\label{sh}
Suppose $u:C\to SH$ is a $J_{\infty}-$holomorphic curve of genus 0 in $SH$   with one positive end asymptotic to a simple Reeb orbit.
Then $u$ is a trivial cylinder.
\lemmatwo

\pfone The proof is contained in   Lemma 7.5 \cite{Evans} (see also \cite{Hind}, \cite{compactness}). We briefly recall the main points.
Since each Reeb orbit is  non-trivial in $\pi_1(H)$ and $C$ has genus 0,
  there has to be at least one negative  puncture.
On the other hand, since $E_{\lambda}(u)\geq 0$  and  all Reeb orbits  have the same period, $u$ has at most one negative puncture, which has to be simple.
 Thus $u$ is  a trivial cylinder.
\pftwo

%In the rest of the paper, $\overline{\JJ}$ always refers to the set
%of almost complex structure defined in (\ref{J-condition}) with the
%choice of $J^0$ specified above.

\subsubsection{$J^0-$holomorphic planes in $T^*S^2$}

In $T^*S^2$ we need to consider embedded holomorphic planes with one (positive) end asymptotic to a simple Reeb orbit.

As mentioned, on $T^*S^2$, $J_{\infty}$ is the same as $J^0$.
Notice that $J^0$ interchanges the two summands of the Lagrangian splitting of the tangent bundle of $T^*S^2$.
Thus $\det (TT^*S^2, J^0)$ is canonically trivialized since the Lagrangian horizontal two plane bundle is orientable.
The expected dimension of the moduli space of embedded $J^0-$holomorphic plane $u$ with one (positive) end asymptotic to a simple Reeb orbit is thus given by
\eqone \label{index=2}\hbox{index}(u)=-2+2+2=2.\eqtwo
This follows from the general index formula \eqref{general index formula}, and the vanishing of
$c_1^\Phi$ for all punctured curves in $T^*S^2$.
% and the computation that the CZ index of Reeb orbits is $2$ relative to $\Phi$.

%\subsubsection{$J_{\infty}^0-$foliations in $T^*S^2$}
It is proved in Lemmas $8$ and $9$ and Section $4$ in \cite{Hind}
  that if $\tilde J^0$ is close to  $J^0$
  %which is SFT regular for
and  any embedded $\tilde J^0-$holomorphic planes with
one simple puncture is regular,  then $\tilde J^0$ enjoys the following properties:

\enone [(1)]
\item There are two $\tilde J^0$-foliations $\mathcal{F}_\alpha$ and $\mathcal{F}_\beta$
in $T^*S^2$, such that there is a one-one correspondence from simple
Reeb orbits to planes in each foliation;
\item Each element in $\fF_\alpha$ ($\fF_\beta$, resp.) intersects the zero-section
at a single point positively (negatively, resp.). \entwo

\nono We will call the planes in $\fF_\alpha$ ($\fF_\beta$, resp.) $\alpha$-planes
($\beta$-planes, resp.).

One consequence of \eqref{index=2} is that we can appeal to Wendl's Theorem \ref{at} to conclude that
%$J^0$ is already SFT regular, at least for
each embedded $J^0-$holomorphic planes with one  simple puncture is regular.
In particular, $J^0$ also satisfies the above properties.
Furthermore, we have

\lemmaone\label{U-plane characterization} A $J^0$-holomorphic
plane in $T^*S^2$ asymptotic to a simple Reeb orbit belongs to either
$\fF_\alpha$ or $\fF_\beta$. Moreover, an $\alpha$-plane and a $\beta$-plane intersect transversally if they do not share the same asymptote.
\lemmatwo

\pfone The proof is largely similar to Lemma $8$ in \cite{Hind}. One
could think of $T^*S^2$ topologically as a neighborhood of
$\bar{\Delta}$, the anti-diagonal in $S^2\times S^2$.  The
complement is then a disk bundle over $\Delta$ the diagonal, of
which the boundary of disk fibers coincides with the simple Reeb
orbits in $T^*S^2$.  One can then glue these disks to elements in
$\fF_\alpha$ and $\fF_\beta$, resulting in two foliations in
$S^2\times S^2$, with classes $[S^2\times pt]$ and $[pt\times S^2]$,
respectively.  Suppose we have a $J^0$-holomorphic plane $P$
in $U$ asymptotic to some simple Reeb orbit $\gamma$, which does not
belong to either $\fF_\alpha$ nor $\fF_\beta$, it must intersect
some $P_\alpha\in \fF_\alpha$ and $P_\beta\in \fF_\beta$ positively,
where $P_\alpha$ and $P_\beta$ have asymptotes
$\gamma_\alpha,\gamma_\beta$ which are different from $\gamma$.  Now
$P$, $P_\alpha$ and $P_\beta$ can all be capped in $S^2\times
S^2$ by the above procedure, resulting in three spheres intersecting
only in $U$.  By construction, the sphere formed by capping $P$ has
positively intersection with both $[S^2\times pt]$ and $[pt\times
S^2]$, but intersects $\Delta$ at a single point, which leads to a
contradiction.

The second assertion can be proved similarly, for if
$\gamma_\alpha\neq\gamma_\beta$, the capped sphere does not have
intersection in the complement of $T^*S^2$, so they must intersect inside
$T^*S^2$ for homological reason.

\pftwo

\rmkone
If we do not appeal to Wendl's automatic transversality result,  instead of $J^0$, we could simply use a fixed $\tilde J^0$ satisfying
the properties above throughout
the paper.
\rmktwo

\begin{comment}
In fact, Richard Hind established the facts for
a small  perturbation of $J^0$ which is SFT regular.
%another adjusted almost-complex structure $\tilde J^0$.
%so that one may
%still perform neck-stretch and achieve transversality.

  As far as the situation
in \cite{Hind} and our paper is concerned, curves under consideration are holomorphic planes free from critical points with virtual index $2$,
so (\ref{Wendl's equation}) clearly holds.  Therefore, Hind's result indeed holds without perturbing ${J}^0$.
%and one could simply
%take $J^0=\tilde{J}^0$.
\end{comment}

\subsubsection{SFT compactness}\label{GH}

 Following \cite{Hind} we briefly recall the relevant compactness results in the symplectic field theory adapted to
our case.  For detailed expositions on the subject,
we refer the readers to \cite{compactness} and \cite{B}.\\

Let $M_\infty=\overline{W}\cup SH\cup\overline{U}$, and $J_\infty$
be the almost complex structure defined as in section \ref{section:Lagrangian and good a.c.s.}.
 Let $\Sigma$ be a Riemann surface with nodes.  A \textit{level-$k$ holomorphic building} consists of the following data:
\begin{enumerate}[(i)]
 \item (level) A labelling of the components of $\Sigma\backslash\{\text{nodes}\}$ by integers $\{1,\cdots,k\}$ which are the \textit{levels}.
Two components sharing a node differ at most by $1$ in levels. Let $\Sigma_r$ be the union of the components of $\Sigma\backslash\{\text{nodes}\}$ with label $r$.

 \item (asymptotic matching) Finite energy holomorphic curves $v_1:\Sigma_1\rightarrow U$, $v_r:\Sigma_r\rightarrow SH$, $2\leq r\leq k-1$, $v_k:\Sigma_k\rightarrow W$.
Any node shared by $\Sigma_l$ and $\Sigma_{l+1}$ for $1\leq l\leq k-1$ is a positive puncture for $v_l$ and a negative puncture for $v_{l+1}$ asymptotic to the same
 Reeb orbit $\gamma$.  $v_l$ should also extend continuously across each node within $\Sigma_l$.
\end{enumerate}

Now for a given stretching family $\{J_{t_i}\}$ as previously described, as well as $J_{t_i}$-curves $u_i:S\rightarrow(M,J_{t_i})$, we define the Gromov-Hofer convergence as follows:

A sequence of $J_{t_i}$-curves $u_i:S\rightarrow(M,J_{t_i})$ is said to be \textit{convergent to a level-$k$ holomorphic building} $v$
 in Gromov-Hofer's sense, using the above notations, if
there is a sequence of maps $\phi_i:S\rightarrow \Sigma$, and for each $i$, there is a sequence of $k-2$ real numbers $t_i^r$, $r=2,\cdots,k-1$, such that:

\enone[(i)]
 \item (domain) $\phi_i$ are locally biholomorphic except that they may collapse circles in $S$ to nodes of $\Sigma$,
 \item (map) the sequences $u_i\circ\phi_i^{-1}:\Sigma_1\rightarrow U$, $u_i\circ\phi_i^{-1}+t_i^r:\Sigma_r\rightarrow SH$, $2\leq r\leq k-1$, and
$u_i\circ\phi_i^{-1}:\Sigma_k\rightarrow W$ converge in $C^\infty$-topology to corresponding maps $v_r$ on compact sets of $\Sigma_r$.
\entwo

\nono Now the celebrated compactness result in SFT reads:

\thmone[\cite{compactness}] \label{cpt} If $u_i$ has a fixed homology class, there is a subsequence $t_{i_m}$ of $t_i$ such that $u_{t_{i_m}}$ converges
to a level-$k$ holomorphic building in the Gromov-Hofer's sense.
\thmtwo

\section{Minimal intersection}\label{minimal-intersection}
In this section we prove Theorems \ref{main} and \ref{general}. There are two main ingredients, the symplectic Seiberg-Witten theory
which produces embedded, connected pseudo-holomorphic submanifolds for a class of compatible almost complex structures  suitable for applying symplectic field theory.
Via neck stretching the symplectic field theory then produces in the limit the desired symplectic  surfaces which intersect $L$ minimally.

\subsection{Embedded and nodal pseudo-holomorphic submanifolds}
We first introduce some notations. All surfaces in this section are closed. Given a class $e\in H_2(M, \ZZ)$, let $\eta_\omega(e)$ be the $\omega-$symplectic genus of
$e$:

\begin{equation} \label{genus(e)}
\eta_{\omega}(e)=\frac{e\cdot e+K_{\omega}(e)+2}{2}.
\end{equation}
This is exactly the genus of a connected embedded $\omega$-symplectic surface
in class $e$ (if there is one) from the adjunction formula.

 Also  define the dimension of $e$

\begin{equation}\label{dim(e)}
d(e)=\frac{-K_\omega(e)+e\cdot e}{2}.
\end{equation}
$d(e)$ is the expected dimension of the moduli space of embedded pseudo-holomorphic curve of genus $\eta_{\omega}(e)$ in the class $e$.
In terms of $\eta_{\omega}(e)$, $d(e)$ can also be expressed as:
$$d(e)=-K_\omega(e)+\eta_\omega(e)-1.$$

 Suppose $C$ is a compact, connected, pseudo-holomorphic submanifold of $M$. Then $C$ has the structure of a Riemann surface and
 it represents a nonzero class $[C]$.
  Moreover,  there is a canonically associated elliptic operator
 \begin{equation}\label{D-C} D_C:\Gamma(N)\to \Gamma(N\otimes T^{1, 0}C),\end{equation}
 where $N$ is the normal bundle of $C$.
$D_C$ is called the normal operator of $C$ and the index of $D_C$ is exactly given by $d([C])$.

Fix a set $\Omega$ of $d([C])$ distinct points. If $\Omega\subset C$, then
we can define the operator
$$D_C\oplus ev_{\Omega}: \Gamma(N)\to \Gamma(N\otimes T^{1, 0}C)\oplus (\oplus_{p\in \Omega}N|_p).$$
The index of $D_C\oplus ev_{\Omega}$ is $0$. And the kernel of
$D_C\oplus ev_{\Omega}$ should be thought of as giving a sort of
Zariski tangent space to the space of pseudo-holomorphic embeddings
of $C$ in $M$ containing the subset $\Omega$ (as a point in the
space of smooth embeddings). $C$ is called $(J, \Omega)$
non-degenerate if the operator $D_C\oplus ev_{\Omega}$ has trivial
cokernel (and also trivial kernel).

$D_C$ is a real CR operator on $(C, N)$. For such operators, there is the following automatic transversality result.

\begin{theorem}[\cite{HLS97}, \cite{IS99}]\label{auto-tran} Let $\Sigma$ be a Riemann surface of genus $g$, and let $L$ be a complex line bundle over $\Sigma$. Let $D$ be a real CR operator.
Suppose $c_1(L)\geq 2g-1$, then $\rm{coker} D=0$.
\end{theorem}

We will show in the next two subsections that in two situations, given a class $e$, there is a Baire set of pairs $(J, \Omega)$ for which there are {\bf connected} $J-$holomorphic
submanifolds of genus $\eta_{\omega}(e)$ through $\Omega$.
The Baire property is shown by  first setting up
 universal models of various type of pseudo-holomorphic curves, and then exploiting the Fredoholm properties of $D$
  in conjunction with the Sard-Smale theorem and the Gromov compactness theorem to rule out unwanted behavior for
  generic pairs $(J, \Omega)$.

We also need to generalize to the case of a nodal pseudo-holomorphic
submanifold in the sense of Sikorav (\cite{Sik}). Let $\Sigma=\cup
\Sigma_i$ be a nodal Riemann surface, where $\Sigma_i$ are the
irreducible components. A $J-$holomorphic map $f:\Sigma\to (M, J)$
is said to be nodal if $f$ has distinct tangents along two branches
at each node. For our purpose, we call a nodal curve $f$
 a nodal submanifold if  $f$ is an embedding on each $\Sigma_i$.
 Thus a nodal submanifold is a union of embedded submanifolds intersecting transversally. Let $C_i=f(\Sigma_i)$.

 For a nodal submanifold, the analogue of \eqref{D-C}, $D_{\cup C_i}$,  is  defined in Section 4 in \cite{Sik} in terms of the normalization of $\Sigma$. $D_{\cup C_i}$ is elliptic and its index is simply given by
 $\sum_i d([C_i])$.

 In this case, for each $i$, fix a subset $\Omega_i\subset C_i$ with $d([C_i])$ distinct points and not containing any of the nodes.
 Then the operator $D_{\cup C_i}\oplus ev_{\cup \Omega_i}$ is an elliptic operator with index zero, and $f$ is called non-degenerate if $D_{\cup C_i}\oplus ev_{\cup \Omega_i}$
 has trivial cokernel.

 The automatic transversality in this context,  Corollary 2 in  \cite{Sik},  implies that $D_{\cup C_i}\oplus ev_{\cup \Omega_i}$ is onto if
  \begin{equation}\label{nodal}  -K_{\omega}([C_i])>0, \quad \hbox{ for each $i$.}
 \end{equation}

%\subsection{Symplectic spheres}
\subsubsection{Symplectic spheres}

Suppose $C$ is an embedded symplectic sphere with self-intersection at least $-1$.
In this case
\eqone\label{genus 0}
d([C])=-K_\omega([C])-1, \quad [C]\cdot [C] = -K_\omega([C])-2.
\eqtwo

The following should be well known. We present some details in view of the generalization to certain configurations, Proposition \ref{symplectic connectedness'}.

\propone\label{symplectic connectedness} Let $(M,\omega)$ be a
symplectic 4-manifold, $e\in H_2(M;\ZZ)$ with $e^2\geq -1$ a class represented by
an embedded symplectic sphere $C$.  Then
 there is a path connected Baire subset $\mathcal T_e$ of $\JJ_\omega\times M_{d(e)}$ such that a pair $(J, \Omega)$ lies in $\mathcal T_e$
 if and only if  there is a unique  embedded $J-$holomorphic sphere in the class $e$ containing $\Omega$.
Here $M_d$ is the space of $d-$trples of distinct (but unlabeled) points in $M$.
 Consequently, any symplectic sphere in the class $e$ is isotopic to $C$.

\proptwo
\pfone

 Pick an almost complex structure $J\in\JJ_{\omega}$ such that
$C$ is $J$-holomorphic and $\Omega\subset C$.

Following \cite{Barraud} (Lemma 4 and formula (15)) and \cite{IS99}, let $P=-\sum_{z_i\in \Omega} z_i$ be the divisor of $C$ and $\tilde N=N\otimes P$.
Then there exists a real CR operator on $(C, \tilde N)$,
$$\tilde D_C:\Gamma(\tilde N)\to \Gamma(\tilde N\otimes T^{1, 0}C),$$
with the property that $\hbox{coker} \tilde D_C\cong \hbox{coker} (D_C\oplus ev_{\Omega})$.
Notice that, by \eqref{genus 0},
$$c_1(\tilde N)=c_1(N)-d([C])=e\cdot e-d(e)=-1.$$
From Theorem \ref{auto-tran}, $\tilde D$ is surjective.

\begin{comment}
We only need to show that $GW(e)=1$.  Fix a set
$\{z_i\}_{i=1}^{d(e)}\in S^2$, from Proposition 3.4.2 of \cite{MSJ},
$ev:\mathcal M(e,S,\JJ)\rightarrow M^{d(e)}$ is transverse at all
points.  Here $\mathcal M(e,S,\JJ)$ is the universal moduli space of
holomorphic spheres in class $e$, see \cite{MSJ} for more details.
Therefore, fix any set of points $\Omega\subset C^{d(e)}\subset
M^{d(e)}$, for an almost complex structure $J\in\JJ_{\omega}$ such
that $C$ is $J$-holomorphic, from Proposition \ref{auto-tran}, $J$
is regular for $C$, thus contributes to the GW-invariant.
\end{comment}

Notice that $d(e)\geq 0$. Moreover,
 from the positivity of
intersections and the fact that $e\cdot e =d(e)-1$, $C$ is the only
connected $J$-sphere in $e$ containing $\Omega$. Since $\tilde D$ is surjective, $C$ is regular with respect to $(J, \Omega)$.
Thus we conclude that the genus 0 Gromov-Witten invariant of $e$ passing through
$d(e)$ points is $\pm 1$, in particular, nonzero.

A marked $\mathbb P^1$ is a pair $(\mathbb P^1, \{z_i\})$ where $\{z_i\}$ is a set of unordered, distinct points.
 Now introduce the universal genus zero moduli space $\mathcal P$ associated to $e$, which is the space of  $J-$holomorphic embedding $u:(\mathbb P^1, \{z_i\}_{i=1}^{d(e)})\to (M, J)$ with $[u]=e$ for some $J\in \mathcal J_{\omega}$, modulo the automorphism of $\mathbb P^1$. $\mathcal P$ is a Frechet manifold (\cite{MSJ}). Moreover, the natural map $\pi$
 to $\mathcal J_{\omega}\times M_{d(e)}$, $(u,J,\{z_i\})\to (J, \{u(z_i)\})$ is Fredholm.
 The argument above simply means that $\pi$ is an isomorphism onto its image.

 Similarly, for each  possible singular type $c$, introduce the auxiliary universal moduli space  $\mathcal P_c$. Each $\mathcal P_c$ is again a Frechet manifold  and the projection $\pi_c: \mathcal P_c\to \mathcal J_{\omega}\times M_{d(e)}$ is Fredholm (\cite{MSJ}) with index at most $-2$.
 Notice that the image of $\pi$ and the union of the images of $\pi_c$ cover $\JJ_{\omega}\times M_{d(e)}$ by the non-triviality of the Gromov-Witten invariant.
 Since each $\pi_c$ has negative index, the complement of the image of $\pi_c$ is exactly the set of regular values of $\pi_c$, hence is Baire.
 This implies the image of $\pi$ is Baire.

 Now we show that the image of $\pi$ is path connected. Let $(J', \Omega')$ be in the image of $\pi$. The Sard-Smale theorem implies that along  a generic path $(J_t, \Omega_t)$ connecting
 $(J, \Omega)$ and $(J', \Omega')$, for each $t$, $(J_t, \Omega_t)$ is either a regular value of projections $\pi$ and $\pi_c$, or it is a singular value for one of the projections but
 the cokernel  has dimension  1. Since each $\pi_c$ has index $-2$ and $\pi$ has no singular values, each $(J_t, \Omega_t)$ lies in $\mathcal T$.
 %We follow the argument in \cite{T1} (Assertion 3 in Proposition 5.2).
%Introduce the space $\JJ_{\omega}(J,\Omega, J', \Omega')$ of paths in $\JJ_{\omega}\times M^{d(e)}$ connecting $(J, \Omega)$ and $(J', \Omega')$. $\JJ_{\omega}(J,\Omega, J', \Omega')$ is a Frechet manifold. $(J_t, \Omega_t\in \JJ_{\omega}(J,\Omega, J', \Omega'), t\in [0,1]$  is called regular if
%Generalize the universal model $\mathcal P$ to include the $t\in [0,1]$ parameter. This
%All non-compactness is described by auxiliary models for which the resulting map has differential with index $-1$.
%Since each such projection has index at most $-2$, each $J_t-$holomorphic curve in the class $e$ through  $\Omega_t$ is embedded.
%Namely the path is in the image of $\pi$.
\begin{comment}
One choose almost complex structures $J_0$, $J_1\in \JJ_\omega$ such that $C_i$
is $J_i$-holomorphic, and $\Omega_i\subset C_i$.  One then choose a generic path connecting $(J_i, \Omega_i)$.
It is generic in the sense that for each possible singular curve type $c$, the path is regular with respect
to the Fredholm projection $\pi_c$ from the corresponding smooth universal moduli space

\end{comment}

Finally, notice that the path connected set $\mathcal T$ maps onto the space of symplectic spheres in the class $e$.

\pftwo

%Following arguments in \cite{Barraud}, one can further show the
%space of embedded symplectic spheres in class $e$ is path connected.

For our application we need to take one  step forward.

\begin{definition} \label{configuration} We call
an ordered  configuration  of symplectic spheres $\cup C_i$
a \textit{stable spherical symplectic configuration} if

1.  $[C_i]\cdot
[C_i]\geq -1$ for each $i$,

2.  for any pair $i, j$ with $i\ne j$, $[C_i]\ne [C_j]$,  and $[C_i]\cdot [C_j]=0$ or $1$.

3. they are simultaneously $J-$holomorphic for  some $J\in \JJ$.

The homological
type refers to the set of homology classes $[C_i]$.
\end{definition}

Notice that, by local positivity of intersection, 2 and 3 imply that $C_i$ and $ C_j$
are either disjoint or
 intersect transversally at one point.
 In particular, it is a $J-$\textit{nodal} submanifold.
 Further, since $C_i\cdot C_i\geq -1$,  the condition \eqref{nodal} is satisfied by \eqref{genus 0}.

If we follow the arguments above, replacing Theorem \ref{auto-tran} by  Corollary  2 in \cite{Sik},  we obtain:

\propone\label{symplectic connectedness'}  Suppose there is a stable spherical symplectic  configuration $\cup_i C_i$ with type $D$.
Then there is a path connected Baire subset $\mathcal T_D$ of $\JJ_\omega\times \prod_i M_{d([C_i])}$ such that a pair $(J, \Omega_i)$ lies in $\mathcal T_D$
 if and only if  there is a unique  embedded $J-$holomorphic $D-$configuration with the $i-$th component  containing $\Omega_i$.
 Consequently,
stable spherical symplectic  configurations
with the same homological type
are isotopic. \proptwo

\subsubsection{Gromov-Taubes invariants when $b^+=1$}\label{section:GT invts}
Given a class $e$ and a pair $(J, \Omega)$ in $\JJ_\omega\times M^{d(e)}$, introduce the set $\mathcal H\equiv \mathcal H(e, J, \Omega)$ whose elements are the unordered sets of pairs $\{(C_k, m_k)\}$ of disjoint, connected, $J-$holomorphic submanifold
$C_k\subset M$ and positive integer $m_k$, which are constrained as follows:

1. If $e_k$ is the fundamental class of $C_k$ then $d_k\equiv d(e_k)\geq 0$.

2. If $d_k>0$, then $C_k$ contains  a subset $\Omega_k\subset \Omega$ consisting of precisely $d_k$ points.

3. The integer $m_k=1$ unless $C_k$ is a torus with trivial normal bundle.

4. $\sum_k m_k e_k = e$.

Notice that \eqref{dim(e)} and \eqref{genus(e)} imply that

\begin{itemize}
\item the only negative square components are spheres with square $-1$;

\item a square $0$ component is either a sphere or a torus;

\end{itemize}

 To define the Gromov-Taubes invariant of a class $e$, a notion of admissibility of pair
 is introduced in \cite{T1}.
 The Gromov-Taubes invariant $GT(e)$ of $e$ is then a suitably weighted count of $\mathcal H(e, J, \Omega)$ for an admissible $(J, \Omega)$, which is
 delicate  at the presence of  a toridal component  with multiplicity higher than 1. When $b^+=1$, we will see that there are simple homological conditions to avoid
such components.

 It is rather involved to fully describe the precise meaning  of admissible pairs, especially at
the presence of a toridal component  with multiplicity higher than 1.
%One basic condition is that
% the normal operator $D_{C_k}\oplus ev_{ \Omega_k}$ of each component $C_k$ is surjective.
In fact, in the case $d(e)=0$, $\Omega$ is the empty set, we are simply talking about the admissibility of $J$ alone. Furthermore, if
 there are no toridal components, $J$ is admissible if $\mathcal H(e, J)$ is a finite set, and each submanifold in a member of $\mathcal H(e, J)$
 is non-degenerate.

 %\enone
%\item
%\item if a component  has multiplicity higher than 1, then it is a sphere with square $-1$, or a torus with square $0$.
%\entwo
%It is shown that in \cite{T1} for each admissible pair,

 It is also shown in \cite{T1} that the set of
  admissible pair is Baire. The argument is similar to the one in  Proposition \ref{symplectic connectedness}.
In fact, by Remark \ref{open set}, the intersection with each $\overline \JJ(i)$ is still Baire in $\overline \JJ(i)$ since
$U$ contains no closed pseudo-holomorphic curve.

When $C$ is a symplectic sphere with
self-intersection at least $-1$,
it is easy to show that $GT([C])=1$ using arguments in Proposition \ref{symplectic connectedness}. In general, when $b^+=1$, due to Taubes'
SW$\Rightarrow$GT \cite{T2} and the Seiberg-Witten wall crossing
formula, there are plenty of classes with non-trival GT invariant, and most of them are represented by connected embedded symplectic surfaces
(\cite{LL}, see also \cite{packing}, \cite{MIsotopy}, \cite{Li}) :

\propone\label{Assumptions when b+=1} Let $(M, \omega)$ be a
symplectic $4$-manifold with $b^+=1$ and canonical class $K_\omega$.
Let $A\in H_2(M;\ZZ)$ be a class satisfying the following
properties:
\begin{itemize}
 \item $A^2>0$ and $\omega(A)>0$;
 \item $A-\text{PD}(K_\omega)$ is $\omega$-positive and has non-negative square;
 \item $A\cdot E\geq0$ for all $E\in\EE_\omega$.
\end{itemize}
Then $A$ has non-vanishing GT invariant and $A$ is
represented by a  connected embedded symplectic surface. \proptwo

%\subsubsection{Positive genus}

\lemmaone\label{gt}
Let $(M,\omega)$ be a symplectic
$4$-manifold with $b^+=1$.
Suppose $e\in H_2(M;\ZZ)$ is a class with $\eta_\omega(e)\geq2$, $e\cdot E\geq0$ for
all $E\in\EE_\omega$, and $GT(e)\neq 0$. Then for any admissible  $(J, \Omega)$, $A$
has a connected $J-$holomorphic representative of genus $\eta_\omega(e)$.
\lemmatwo

\pfone
Suppose $(J, \Omega)$ is admissible.
Let $C$ be  a $J-$holomorphic submanifold contributing to $GT(e)$. The condition that $e\cdot E\geq0$ for all $E\in\EE_\omega$ ensures $C$ has
no negative-square components.  Since $b^+(M)=1$, if $C$ is  disconnected, then all the components are homologous and have square $0$.
Thus $C$ is either  a union of spheres with square $0$, or a union of tori with square $0$.  However,  this contradicts the assumption that
$\eta_\omega(e)\geq2$ from the adjunction formula.  Therefore $C$ is a connected  genus $\eta_{\omega}(e)$ surface as claimed.

\pftwo

Furthermore, assume that $d(e)\geq 1$.  Let $\{U_i\}_{i=1}^{d(e)}$
be a sequence of pairwisely disjoint Darboux chart. We consider the
class of almost complex structures $\mathcal{F}_{\{U_i\}}\subset
{\JJ}_{\omega}$ which is fixed  and integrable on $U_i$.
%In other words, for $J\in\mathcal F_{\{U_i\}}$, $J|_{U_i}$ is
%a fixed integrable complex structure  compatible with $\omega$, which induces a K\"ahler
%metric  isometric to a standard ball in $\CC^2$.
By Remark
\ref{open set}, there is an admissible pair $(\tilde J, \tilde \Omega=\{x_i\})$
with $\tilde J\in \mathcal{F}_{\{U_i\}}$ and $x_i\in U_i$. In particular,
there is a connected  embedded $\tilde J-$holomorphic curve $\tilde C$ through $\{x_i\}$
with $[\tilde C]=e$.

 For any such $\tilde{J}\in \mathcal F_{\{U_i\}}$, let  $p:(M', \tilde{J}_{\{x_i\}})\to (M, J)$
 be the complex blow-up of $(M, \tilde{J})$ at $x_i$.
 Denote each exceptional sphere by $C_{x_i}$ and its neighborhood corresponding to $U_i$ by $U'_{i}$.
 One can then endow $M'$ with a symplectic form $\omega'$ compatible with $J_{\{x_i\}}$.
% such that the area of the $C_{x_i}$ are tiny
(see Lemma 7.15 in \cite{MSI}).
 Denote also $\mathcal F'_{\{U_i\}}\subset \JJ_{\omega'}$ to be the corresponding
 set of almost complex structures.

\lemmaone\label{lemma:blown-up curves, GT genericity} Given the same
assumption in Lemma \ref{gt} and consider $(M',J_{\{x_i\}})$ as
above. Let $E_i=[C_{x_i}]$, $i=1,\cdots,d(e)$ and $A'=A-\sum_{1\leq i\leq d(e)}E_i$. Then
$GT_{\omega'}(A')\ne 0$, and for $J$ in a Baire subset of
$\mathcal F_{\{U_i\}}$, $A'$ is represented by a connected
$J_{\{x_i\}}-$holomorphic surface of genus $\eta_{\omega}(A)$, intersecting each $C_{x_i}$ transversally
at one point.
\lemmatwo

 \pfone

Now $-K_{\omega'}(A')=-K_{\omega}(A)-d(e)$ and
$\eta_{\omega'}(A')=\eta_{\omega}(A)$.  Since $d(A')=0$, from the blow-up formula,
Corollary 4.4 in \cite{LL2}, $A'$ also has nontrivial GT
invariant.

Since the only $J'-$holomorphic curves contained in $\cup U'_{i}$
are  $C_{x_i}$, by Remark \ref{open set}, the intersection of
admissible almost complex structures on $(M', \omega')$ with
$\mathcal F'_{\{U_i\}}$ is a Baire set in $\mathcal F'_{\{U_i\}}$.

To check the generic connectedness, by Lemma \ref{gt}  we only need
to verify the homological condition $A'\cdot E\geq0$ for
any $E\in \EE_{\omega'}$. But as is shown above, there is  a
connected embedded $\tilde J-$holomorphic $\tilde C\subset M$, thus its proper
transformation $\tilde C'\subset M'$ is $J_{\{x_i\}}$-holomorphic with
$[\tilde C']=A'$. Notice also that every $E$ has a
$J_{\{x_i\}}-$holomorphic representative since exceptional classes
always have non-trivial GW invariant. Since the genus of $\tilde C'$ is
positive, it is different from any component of $E$. By positivity
of intersections, we have $[\tilde C']\cdot E\geq 0$.

Finally, since $\mathcal F'_{\{U_i\}}$ and $\mathcal F_{\{U_i\}}$ are canonically identified
via complex blowing up the $x_i$ and complex blowing down the $C_{x_i}$, we obtain
the required Baire subset of $\mathcal F_{\{U_i\}}$.

\pftwo

\subsection{Proof of Theorems \ref{main} and \ref{general}}

We are ready to prove Theorem \ref{main} and \ref{general}. For the
convenience of exposition, we first investigate the behavior of
generic $J$-holomorphic representatives in class $A$ in
neck-stretching, when the class $A$ satisfies

\eqone\label{equality}  -K_\omega(A)=1-\eta_\omega(A).\eqtwo

Firstly, regarding the fixed class $A$, we claim that there is a
Baire set $\JJ_{reg}(A)\subset \overline \JJ$ such that for each $J \in
\JJ_{reg}$, $J_{t_i}$ is GT admissible for each $i$, and $J_{\infty}$
is regular in the sense of SFT for $\overline W$.
By Proposition \ref{evans} there is a Baire subset
$\JJ_{reg}'\subset \overline \JJ$, such that for $J\in\JJ'_{reg}$,
$J_\infty$ is SFT regular.
Recall from  \ref{section:good almost complex structures},
$\overline \JJ(i)=\{J|J=J^0_{t_i}\text{ in }U\}$ and $P_i$ is the identification of $\overline \JJ(i)$ with $\overline \JJ$.  We have mentioned that, as all
closed pseudo-holomorphic curves have to pass through $M\backslash
U$,   there is a
 Baire subset $\overline \JJ(i)'\subset \overline \JJ(i)$
such that each member is GT admissible.  One then takes
$\JJ_{reg}(A)=\cap_n P_n(\overline \JJ_n') \cap \JJ_{reg}'$.
%Let $\overline \JJ(i)'=\{J\in \overline \JJ| J \in \overline \JJ(n)'\}$ be the corresponding Baire subset in $\overline \JJ$.

Fix $J\in \mathcal J_{reg}(A)$. By Lemma \ref{gt} there is a sequence
of connected embedded $J_{t_i}$-holomorphic submanifolds $C_{t_i}$.
 %One is referred to \cite{compactness} for a comprehensive description of neck-stretching in SFT.
If $C_{t_i}$ does not intersect $L$ for some $i<\infty$,
 the theorem follows.  Now we assume
  that each $C_{t_i}$ intersects $L$.  This assumption will
 eventually lead to  a contradiction when $[L]\cdot [C]=0$ and is automatically
  satisfied if $[L]\cdot [C]\neq 0$.

By Theorem  \ref{cpt}, there is  a
$k-$leveled curve $C_\infty$ as a Gromov-Hofer limit of
$\{C_{t_i}\}_{i=0}^\infty$: the piece in $M\backslash U_l$,
which we call $C_W$ or the $W$-part; the piece in the symplectization of
$\partial U_l=\RR P^3$ consisting of $k-2$ levels, which we call $C_{SH}$ or the $SH$-part;  the piece in
$U_l$, which we call $C_U$ or the $U$-part.
Let us first examine the $W$-part.

\lemmaone\label{simple punctures} Suppose \eqref{equality} is satisfied. Then $C_W$ is a, possibly unbranched
covering, irreducible genus-$\eta_\omega(A)$ curve, and all
asymptotic Reeb orbits are simple.  Moreover, let $\bar{C}_W$ be the
underlying simple curve, then the limits of punctures of $\bar{C}_W$
are pairwisely distinct. \lemmatwo

\pfone
%Since we assume that $C_t$ intersects $L$, the $U$-part of
%$C_\infty$ is non-empty.
By the maximum principle, $C_W$ is non-empty. Let $u_i:B_i\to W,
1\leq i\leq q,$ be the irreducible components of  $C_W$ and $g_i$
the genus of $B_i$.  Suppose $u_i$ is a degree $m_i$ multiple cover of $\bar u_i:
\bar B_i\to W$.

Notice that
$$c_1^\Phi=0 \text{ in } U \text{ and } S$$
implies that
\eqone \label{ui}  \sum_{1\leq j\leq q} c_1^\Phi(TW)([u_j])=-K_{\omega}(A).\eqtwo
From the description of Gromov-Hofer convergence in \ref{GH}, we clearly have $\sum_{1\leq j\leq q} g_j\leq
\eta_{\omega}(A)$.  \eqref{equality} then implies that
$$\sum_{1\leq j\leq q} c_1^\Phi(TW)([u_j])\leq 1-\sum_{1\leq j\leq q} g_j.$$

If $q>1$, there must be some component, say $B_1$, with

\eqone\label{new} c_1^\Phi(TW)([u_1])\leq -g_1.\eqtwo
%Suppose $\bar s_1$ is the number of punctures of $\bar u_1$ converging to $\bar \gamma_i$, $i=1,..., \bar s_1$.
By \eqref{general index formula}, we have
 \eqone\label{index formula}
 \text{index}(\bar u_1)=-(2-2g(\bar B_1))+2\bar s_1^-+2c_1^\Phi(TW)([\bar u_1])-\sum^{\bar s_1^-}_{k=1}2\hbox{cov}(\bar \gamma_{k})
  \eqtwo

\nono Here $\bar s_1^-$ is the total number of punctures of $\bar
u_1$ and the $\bar \gamma_k$ are the asymptotic Reeb orbits. By our
choice of $J$,  index$(\bar u_1)\geq 0$, thus  we must have \eqone
\label{bar} c_1^\Phi(TW)([\bar u_1])\geq 1-g(\bar B_1).\eqtwo Notice
that $c_1^\Phi(TW)([u_1])=m_1c_1^\Phi(TW)([\bar u_1])$. Since $2\bar
s_1^--\sum^{\bar s_1^-}_{k=1}2\hbox{cov}(\bar \gamma_{k})\leq 0$,
by (\ref{new}) we have  $$g_1\leq m_1(g(\bar B_1)-1).$$
But this is impossible by the Riemann-Hurwitz formula

\eqone\label{rh} (g_1-1)\geq m_1(g(\bar B_1)-1).\eqtwo

\nono This contradiction  shows that $C_W$ is  irreducible, namely,
given solely by $u_1$, when $J\in\JJ_{reg}$.
 By  (\ref{bar})  and \eqref{ui}, we have
 $$1-\eta_{\omega}(A)=m_1c_1^\Phi(TW)([\bar u_1])\geq m_1(1-g(\bar B_1)) .$$ Since $g_1\leq \eta_{\omega}(A)$, we have by
 \eqref{rh}, that $$\eta_{\omega}(A)=g_1.$$
Notice that this also means  $u_1$ is an unbranched covering.
Now return to \eqref{bar}, we find that

\eqone\label{index}(\bar u_1)=2\bar s_1^--\sum^{\bar
s_1^-}_{k=1}2\hbox{cov}(\bar \gamma_{k})\geq 0.\eqtwo

Hence we conclude that each $\bar \gamma_k$ is a simple Reeb orbit.
Since $u_1$ is an unbranched covering, each of its puncture also
converges to one of the simple Reeb orbits,  $\bar \gamma_k$.

One also sees from (\ref{index formula}) and (\ref{index}) that
$C_W$ must have genus $g$ and all asymptotes are simple.

Since the Reeb orbits form a two dimensional Morse-Bott family, the
last statement follows from the transversality of puncture
evaluation of $\bar{C}_W$ (Theorem 5.24 \cite{Evans}). \pftwo

Now we look at the  $S$-part $C_{SH}$.

\lemmaone\label{lemma:trivial in symplectization} Each component of $C_{SH}$ is a
trivial cylinder asymptotic to a simple Reeb orbit. \lemmatwo

\pfone
$C_{SH}$ has $k-2$ levels. Let $\tau_i:D_i\to SH$ be an
irreducible component of first level of $C_{SH}$. Since $C_W$ is connected and already
has genus $\eta_{\omega}(A)$, $D_i$ is of genus $g$ and has a unique positive
puncture since the domain of  $C_{\infty}$ is obtained by collapsing a genus $g$ surface. Moreover, due to the asymptotic matching condition between two levels,
this unique positive puncture of $\tau_i$ is asymptotic to a simple Reeb orbit since all the asymptotes of $\bar{C}_W$ are simple.
Thus  $\tau_i$ must be a trivial cylinder by Lemma \ref{sh}. Similarly  each component in higher level of $C_{SH}$ must be a cylinder as well
(In fact,  there can only be one level of trivial cylinders in $C_{SH}$ by  the finite automorphism requirement  of $C_{\infty}$, but we do not need this more precise description).
\pftwo

For the $U$-part $C_U$, in turn,
%Lemma \ref{simple punctures} and
Lemma \ref{lemma:trivial in symplectization} implies that  all the positive
punctures of $C_U$ are simple due to the asymptotic matching condition between two levels. Moreover, each component $F_i$ is of  genus 0
and has only one positive puncture, again due to the constraint
$g(C_{\infty})= g$. Thus each $F_i$ is a plane with one simple
positive puncture. From Lemma \ref{U-plane characterization} and
Lemma \ref{simple punctures}, the $U$-part is a union of some
$\alpha$- and $\beta$-planes.

\lemmaone\label{U-planes are in the same family} If $C_W$ is not a
multiple cover, the $U$-part consists of either all $\alpha$-planes
or all $\beta$-planes. \lemmatwo

\pfone
The proof is similar to \cite{Evans} Lemma 7.8.
As is explained in Lemma \ref{U-plane characterization}, an
$\alpha$-plane and a $\beta$-plane do not intersect only if they
have the same asymptotic Reeb orbit.  This must be the case to avoid
self-intersection of the holomorphic building $C_\infty$ which
contradicts the embeddedness for $C_{t_i}$ at some $i<\infty$.

Therefore, if the $U$-part has at least one $\alpha$-plane and one
$\beta$-plane, all planes must asymptote to the same Reeb orbit.
If $C_W$ is not a
multiple cover, since the $C_{SH}$ part consists of trivial cylinders, this is impossible by the last statement of
Lemma \ref{simple punctures}.

\pftwo

\begin{comment}
Now the minimal intersection statement follows from Lemma \ref{U-planes are in
the same family}.  For example, when $[L]\cdot [C]=k>0$,
$C_\infty$ intersects $L$ transversally at $k$ points via
$\alpha$-planes, hence positively.  Therefore for some $i<\infty$,
the $C_{t_i}$ satisfies all requirements of our symplectic curve.  The
cases for $[L]\cdot [C]=k\leq0$ similarly follows.  This
concludes the proof of Proposition \ref{technical}.

\pftwo
\end{comment}

\pfone[Proof of Theorem \ref{general}:] It is straightforward that
when $n\in\mathbb{N}$ is large, under the assumption of Theorem
\ref{general} the multiple class $nA$ has the following properties:  $d(nA)>0$,  $GT(nA)\ne 0$, and it
is represented by a connected symplectic surface with genus
at least $2$.

We adapt Welschinger's idea in \cite{Wel} and adopt the notations in Section \ref{section:GT invts} here.  Choose  Darboux
charts $U_i\subset W$, $i=1,\cdots,d(nA),$ and consider $\mathcal
F_{\{U_i\}}$ as in the paragraphs preceding  Lemma
\ref{lemma:blown-up curves, GT genericity}. Now choose $x_i\in U_i$ and an arbitrary $J\in \mathcal F_{\{U_i\}}$,
%  let  $p:(M',J_{x_i})\to (M, J)$, $C_{x_i}$, $E_i=[C_{x_i}]$, and
$A'=nA-\sum_{1\leq i\leq d(e)} E_{i}$  as in Lemma
\ref{lemma:blown-up curves, GT genericity}.
%$A'=A-\sum_{1\leq i\leq d(e)} E_{i}$
By Lemma \ref{lemma:blown-up curves, GT genericity} and the
arguments in the  paragraph following \eqref{equality}, there is a Baire set $\JJ_{reg}(nA)\subset
\overline \JJ\cap \mathcal F_{\{U_i\}}$
 such that for each $J \in \JJ_{reg}(nA)$, $(J_{\{x_i\}})_{t_j}$ is GT admissible for each $j$ and there is a connected
  embedded $(J_{\{x_i\}})_{t_j}$-holomorphic curve $C'_j$ in the class $A'$.
Moreover,  $(J_{\{x_i\}})_{\infty}$ is regular in the sense of SFT
for the symplectic completion of $p^{-1}(W)$.

Now let us analyze
the limit building $C'_\infty$.

Notice that $-K_{\omega}(A')=1-\eta_{\omega}(A')$, so Lemma \ref{simple punctures} could be applied. Also, from the
fact that $C_{x_i}\cap U=\emptyset$ and $A'\cdot E_i=1$, we have $C_{x_i}\cap C'_{W}=1$.  Therefore the
$W$-part of $C'_\infty$ cannot be a multiple cover. Therefore, by Lemma
\ref{U-planes are in the same family},
 $C'_\infty$  intersects $L$ transversally
at finitely many points, where either all the local intersections
are positive or all of them are negative.
This implies that
for some
$j<\infty$, there is an embedded $(J_{\{x_i\}})_{t_j}$-holomorphic
curve $C'_{t_j}$ in the class $A'$ with the same intersection property.
Notice that $C'_{t_j}$ intersects transversally with each $C_{x_i}$ at one point. One then obtain the
desired curve $C$ in the class $A$ by complex blowing down the
(disjoint) exceptional curves $C_{x_i}$.

\pftwo

\rmkone
When $\kappa=-\infty$, given $A$ in Theorem \ref{general}, we can actually find a symplectic surface intersecting
$L$ minimally in the class $A$, rather than $nA$ for large $n$, if we further assume that $\eta(A)\geq 2$ and $A^2\geq \eta(A)-1$.
Here $\eta(A)$ is the symplectic genus (see Section \ref{section:K-Lagrangian}).
This is because, by \cite{Li^2} one
could achieve the non-triviality of GT invariants as long as
$A^2\geq \eta(A)-1$.   And if the class $A$ is reduced
(see Section \ref{section:K-Lagrangian} for the rational case and
\cite{Li^2} the general case), one only needs easily verified
 conditions $\eta_\omega(A)\geq 2$ and  $A^2\geq \eta_\omega(A)-1$ since
$\eta_\omega(A)=\eta(A)$.
%However, when $\eta_\omega(A)>0$ in this
%case, we no longer have the assertion of symplectic isotopy but only
%the existence of such a surface.

\rmktwo

\pfone[Proof of Theorem \ref{main}:] We first deal with the case of
$-1$ sphere $C$. By Proposition \ref{symplectic connectedness},
there is a Baire set $\JJ_{reg}([C])\subset \overline \JJ$ such that
for each $J \in \JJ_{reg}([C])$, there is a unique embedded
$J_{t_i}-$holomorphic sphere in the class $[C]$  for each $i$, and
$J_{\infty}$ is regular in the sense of SFT for $\overline W$.
Notice that $d([C])=0$, and since $C$ has genus $0$, its $W$-part
under neck-stretching does not admit a non-trivial unbranched cover. Therefore we can apply
Lemma \ref{U-planes are in the same family}  as in the proof of Theorem \ref{general} to produce a
$J_{t_i}-$holomorphic sphere $C_{t_i}$ intersecting $L$ minimally.
$C_{t_i}$ is symplectic isotopic to $C$ by the last statement of
Proposition \ref{symplectic connectedness}.

For a symplectic sphere $C$ with non-negative square, we follow the strategy above by  first introducing $U_i$  and $\mathcal F_{U_i}$.
By applying Remark \ref{open set} and Proposition \ref{symplectic connectedness} to $M$ and $[C]$,
there is a pair  $(\tilde J, \tilde \Omega=\{x_i\})$ with $\tilde J\in \mathcal{F}_{\{U_i\}}$, $x_i\in U_i$,
and an embedded $\tilde J-$holomorphic sphere $\tilde C$ through $\{x_i\}$ with $[\tilde C]=[C]$.
Let  $(M',J_{\{x_i\}}, \omega')$, $C_{x_i}$,  $E_i=[C_{x_i}]$,  $U'_i$,   $\mathcal F'_{U_i}$, $i=1,\cdots,d(e)$  be as in Lemma
\ref{lemma:blown-up curves, GT genericity}. The class  $A'=[C]-\sum_{1\leq i\leq d(e)}E_i$
 is represented by an $\omega'-$symplectic $-1$ sphere, for instance, the proper transform of $\tilde C$, thus Proposition
\ref{symplectic connectedness} still holds for $A'$.

Now apply Remark \ref{open set} to $M'$ and $A'$, then Proposition \ref{symplectic connectedness}
% for $J$ in a Baire subset of
%$\mathcal F_{\{U_i\}}$, $A'$ is represented by a unique embedded
%$J_{\{x_i\}}-$holomorphic sphere $\tilde C'$, intersecting each $C_{x_i}$ transversally
%at one point.
and the arguments in the first paragraph of the present subsection imply that,
%and Proposition \ref{symplectic connectedness},
 there is a Baire set $\JJ_{reg}([C])\subset \overline \JJ\cap \mathcal F_{\{U_i\}}$
 with the following property: for each $J \in \JJ_{reg}([C])$, there is a unique embedded $(J_{\{x_i\}})_{t_j}-$holomorphic sphere
$C'_{t_j}$ in the class $A'$   for each $j$,
and  $(J_{\{x_i\}})_{\infty}$ is regular in the sense of SFT for the symplectic completion of $p^{-1}(W)$.
Moreover, $C'_{t_j}$ intersects transversally with each $C_{x_i}$ at one
point.
\begin{comment}
Notice that  $-K_{\omega}(A')=1-\eta_{\omega}(A')$, we can follow
the above argument of neck-stretching to conclude that, for some
$j$, there is an embedded $(J_{\{x_i\}})_{t_j}$-holomorphic sphere
$C'_{t_j}$ in class $A'$, which intersects $L$ minimally. Notice
that $C'_{t_j}$ intersects transversally with each $C_{x_i}$ at one
point. One then obtain the desired sphere $p(C'_{t_j})$ in class $A$
by complex blowing down the (disjoint) exceptional curves $C_{x_i}$.
\end{comment}
Now, just as in the end of the proof of Theorem \ref{general}, for some $j$, $p(C'_{t_j})$
is the desired symplectic sphere in the class $A$, where $p:M'\to M$ is the complex blowing down map.
Moreover,  $p(C'_{t_j})$ is symplectic isotopic to $C$ by the last
statement of Proposition \ref{symplectic connectedness}.

\pftwo

\rmkone\label{Two Lagrangians} One easily sees that the above proof
works for finitely many Lagrangian spheres that are pairwisely
disjoint. It is not clear to the authors whether the theorem holds
when they do intersect. \rmktwo

On the other hand, by choosing subsequences succesively, one may push off certain symplectic
configurations. In particular,   the following will be used in the proof of Theorem \ref{Lagrangian isotopy uniqueness}.

\corone\label{lemma:configuration existence} Let $L$ be a Lagrangian
sphere in a symplectic $4$-manifold $(M,\omega)$, and
$D=\{A_1,\cdots,A_n\}$ a homology type of a stable spherical symplectic configuration.
If each $A_i$ pairs trivially with $[L]$. Then there is a symplectic $D-$configuration
disjoint from $L$.

\cortwo

\rmkone Further,  we expect to  be able to deform a contractible family of
symplectic spheres to be disjoint from a given Lagrangian sphere. Such a result would be useful
in proving Conjecture \ref{k-blowups} on the uniqueness up to symplectomorphism (see Remark \ref{New Connectedness Theorem} and  \ref{us}).
% in a
%weak sense: one needs a family of $J_T$-spheres to start with, where
%$T$ is a parameter of almost-complex structures.
The family being contractible is necessary: as  pointed out to us by R.
Hind, if one takes a representative of the generator of
$\pi_1(Symp(S^2,\sigma))$, the graph of this generator as a circle family
of symplectic spheres in a monotone $S^2\times S^2$ cannot be isotoped away from
 the antidiagonal.  \rmktwo

\section{$K$-null spherical
classes when $\kappa=-\infty$}\label{section:K-Lagrangian}

It is in general difficult to determine whether a spherical class has a Lagrangian spherical representative. We are able to completely solve
 this problem for rational and ruled manifolds in Section \ref{section:existence}.
 In this section we first derive some preliminary results.

\subsection{Rational manifolds}
We fix some notations: in this section $M$ is $\CC P^2\#
n\overline{\CC P}^2$ with $n\geq 1$.
Let $\EE$ and $\LL$  be the
sets of integral homology classes represented by smoothly embedded
spheres of square $-1$ and $-2$ respectively.

An
orthogonal basis $\{H,E_1,\cdots ,E_n\}$ of $H_2(M;\ZZ)$ is called standard if  $H^2=1$ and
$E_i\in \EE$.
We fix a standard basis in this section.

Let $\KK$ be the set of symplectic canonical classes of $M$.
For any sequence $\{\delta_i\}, i=0,..., n$ with $\delta_i=0$ or $1$,
let $K_{\{\delta_i\}}$ be the Poinc\'are dual of
$$-3H+(-1)^{\delta_1} E_1+(-1)^{\delta_2}
E_2-\cdots+(-1)^{\delta_n} E_n.$$
Then $K_{\{\delta_i\}}\in \KK$. When $\delta_i=0$ for any $i$, we simply  denote
it by $K_0$, i.e.
$$K_0=PD(-3H+E_1+\cdots + E_n).$$

 \subsubsection{$\EE, \LL$, symplectic genus and $D(M)$}
We review some  facts about $\EE, \LL$, $D(M)$ and the notion of symplectic genus.

Let $D(M)$ be the geometric automorphism group of $M$, i.e.  the image of the diffeomorphism
group of $M$ in $Aut(H_2(M;\ZZ))$. We say two classes in
$H_2(M;\ZZ)$ are \textit{equivalent} if they are related by $D(M)$.

 It is shown in \cite{Li^2} that $D(M)$ is generated by a set of spherical  reflections.
For $\gamma\in H_2(M;{\mathbb Z})$
with $\gamma^2=\gamma\cdot\gamma=\pm 1$ or $\pm 2$, there is an
automorphism $R(\gamma)\in \text{Aut}(H_2(M; \mathbb Z))$  called the reflection along
$\gamma$,
$$R(\gamma)(\beta)=\beta-
\frac{2(\gamma\cdot\beta)}{\gamma\cdot\gamma}\gamma.$$
If  $\gamma\in \EE$ or $\LL$, by Proposition 2.4 in Chapter III in \cite{FM1},
 $R(\gamma)\in D(M)$, and we call it a spherical reflection.

Another fact is that   $D(M)$ acts transitively on $\KK$ (\cite{LL}).

To define the symplectic genus of  $e\in H_2(M;\ZZ)$ first introduce
the subset $\KK_e$ of $\KK$:

$$\KK_e=\{K\in\KK|\text{there is a class}\hskip 1mm \tau\in\CCC_K\hskip 1mm
\text{such that}\hskip 1mm \tau\cdot e>0\}.$$
Here $\CCC_K=\{[\omega]|\omega\text{ is a symplectic form, }K_{\omega}=K\}$ is the $K$-symplectic cone.
%$$\KK(e)=\{K\in\KK_e|K\cdot e\geq K'\cdot e\hskip 1mm \text{for any}\hskip 1mm K'\in\KK_e\}$$
It is shown in \cite{LL} that
$\CCC_K$ is completely determined
by  the set of $K$-exceptional spherical classes
%\begin{eqnarray}
$$\EE_K=\{E\in\EE|K(E)=-1\}.$$
%\end{eqnarray}
More precisely,
$$\CCC_K=\{\tau\in H^2(M;\RR)|\tau^2>0, \tau(E)>0\hskip 1mm \text{for any}\hskip 1mm E\in\EE_K \}.$$

The following is a useful observation.
\lemmaone\label{delta} If  $\xi=aH-\sum b_iE_i\in H_2(M;\ZZ)$ with  $a>0$
then $K_{\{\delta_i\}}\in \KK_\xi$.
\lemmatwo

\pfone

Notice that for any $K_{\{\delta_i\}}$, one could easily find   $\tau\in\CCC_{K_{\{\delta_i\}}}$  by
requiring  $\tau(H)\gg0$, but keeping the corresponding signs
of $E_i$ in $\tau$ opposite to that of $K_{\{\delta_i\}}$. Such a
construction follows from the easy observation that classes in
$\EE_{K_{\{\delta_i\}}}$ are obtained by changing the corresponding
signs of those in $\EE_K$ and Theorem 4 of \cite{LL}.

By possibly even enlarging $\tau(H)$ further, since $a>0$, one
could also assure that $\tau(\xi)>0$.  Therefore,
$K_{\{\delta_i\}}\in\KK_\xi$.
\pftwo

For $K\in \KK_e$ define the $K-$symplectic genus $\eta_K(e)$  to be
$\frac{1}{2}(K(e)+e^2)+1$.
Finally, the \textit{symplectic genus} of
class $e$ is defined as:

$$\eta(e)=\max_{K\in\KK_e} \eta_K(e). $$

\nono

By  Lemma 3.2 in \cite{Li^2},   $\eta(e)$ has the following
basic properties:

\enone[(1)]

\item $\eta(e)$ is no bigger than the minimal genus of $e$, and
they are both equal to $\eta_{\omega}(e)$ in \eqref{genus(e)}  if $e$ is represented by an $\omega-$symplectic surface for some symplectic form $\omega$;

%\item $\eta(e)=\eta(-e)$

\item Equivalent classes have the same $\eta$.

%\item When $e\cdot e\geq -1$, $\eta(e)=0$ is equivalent to $GT(e)$ is non-trivial.
\entwo

Note that in \cite{Li^2} these properties are stated for
classes with positive square, but the proof actually covered all cases.

\nono We have the following assertions characterizing   $\EE$ and $\LL$ in terms of the symplectic genus,
 as well as the  action of $D(M)$ on $\EE$ and $\LL$.

\propone[\cite{Li^2}, Lemma 3.4, Lemma 3.6(2)]\label{symplectic genus
0=non-reduced}
For $e$ with $e\cdot e=-1$ or $-2$, $\eta(e)=0$ if and only if $e$ is not
equivalent to a reduced class.

Moreover, for $e$ with $e\cdot e=-1$, $\eta(e)=0$ if and only if
$e\in \EE$, Any class in $\EE$ is equivalent  to either $E_i$ or
$H-E_i-E_j$ for some $1\leq i,j\leq n$. If $n\neq 2$,  it is
equivalent to  $E_i$.

Similarly, for $e$ with $e\cdot e=-2$, $\eta(e)=0$ if and only if
$e\in \LL$. Any class in $\LL$ is equivalent  to either $E_i-E_j$ or
$H-E_i-E_j-E_k$ for some $1\leq i,j,k\leq n$. If $n\neq 3$,  it is
equivalent to  $E_i-E_j$. \proptwo

Here a class
$\xi=aH-\sum_{i=1}^n b_iE_i$ with $a\geq 0$ and $b_1\geq b_2\geq\cdots\geq b_n\geq 0$ is called \textit{reduced} (\cite{Gao}, \cite{Ki}) if
$$a\geq b_1+b_2+b_3.$$
%$$ \left\{\begin{aligned} &b_1\geq b_2\geq\cdots\geq b_n\geq 0\\
%                        &a\geq b_1+b_2+b_3 \end{aligned}\right.$$

\subsubsection{$K-$null spherical classes and $D_K(M)$}

For $K\in \KK$ let $D_K(M)$ be the isotropy subgroup of $K$ of the transitive action of $D(M)$ on $\KK$.
We say two classes are \textit{$K-$equivalent} if they are related
by $D_K(M)$.

By definition \ref {def:K-Lag spherical},  $\xi\in H_2(M;\ZZ)$ is  a
$K$-null spherical class if $\xi\in \LL$ and $K(\xi)=0$.
Hence the set of $K-$null spherical classes is denoted by $\LL_K$.
%\begin{enumerate}[(1)]
% \item $\xi$ is represented by a smooth embedded sphere,
% \item $\xi^2=-2$,
% \item $K(\xi)=0$.
%\end{enumerate}

We now study the interactions of $\LL_K$ and $D_K(M)$.
Due to the transitivity of the  action of $D(M)$ on $\KK$ (c.f. \cite{LL}), we
will restrict to the case $K=K_0$ without loss of generality.

First of all,   if $\gamma\in \LL_{K_0}$, then $R(\gamma)\in D_{K_0}(M)$, and we call it a \textit{$K_0$-twist}.

Secondly, notice that $D_{K_0}$ acts on $\EE_{K_0}$ and $\LL_{K_0}$.

To go further,
we need to understand $\EE_{K_0}$.
It is clear that $E_i\in \EE_{K_0}$. Moreover, for any symplectic form $\omega$ with $K_{\omega}=K_0$,
the GT invariant of $H$ or any $E\in \EE_{K_0}$ is non-trivial.
By the positivity of intersection, we have

\lemmaone\label{-1}
Suppose $\xi=aH-\sum b_iE_i$ is in $\EE_{K_0}$, then $a\geq 0$ and $b_i\geq 0$. If $a=0$, then $\xi=E_i$ for some $i$.
\lemmatwo

It is clear that
reflections $R(E_i-E_j)$ and $R(H-E_i-E_j-E_k)$ are $K_0$-twists.
With this understood, we see that  Proposition 1.2.12 in \cite {MSE} can be stated as follows,

\propone \label{exceptional sphere classification}
 Any class in $\EE_{K_0}$ can be transformed to either
$E_i$ or $H-E_i-E_j$ for some $1\leq i,j\leq n$ via $K_0$-twists. If
$n\neq 2$,  it is $K_0-$equivalent to  $E_i$ via $K_0$-twists.
\proptwo

\nono As a consequence, we have

\corone\label{base extension} Suppose
$b^-(M)=n\geq 2$.  If $\{E'_i\}_{i=1}^k$, $k\leq n-2$, is an orthogonal subset of
$\EE_{K_0}$, then there exists $\phi\in D_{K_0}(M)$, generated by $K_0$-twists, such that
$\phi(E_i')=E_i$, $1\leq i\leq k$.
\cortwo

\pfone
The statement is vacuous if $n\leq 2$ and   easily verified for $n=3$. We apply induction on $n$.
From Proposition \ref{exceptional sphere classification}, there exists $\tilde\phi\in D_{K_0}(M)$ such that
$\tilde\phi(E'_1)=E_i$.  One then further compose the $K_0$-twist $f=R(E_i-E_1)$ so that $E_1'$ is eventually sent
to $E_1
$. Noting that
$$f(\tilde\phi(E_i'))\cdot E_1=f(\tilde\phi(E_i'))\cdot f(\tilde\phi(E_1'))=E_i'\cdot E_1'=-\delta_{1i},$$ we are reduced to
the case $n-1$ by restricting our attention to the last $(n-1)$ exceptional
classes (and $k$ is reduced by $1$ as well).
\pftwo

\rmkone Note that this is not true when $k=n-1$.  Take $n=2$. Then $H-E_1-E_2$ is not equivalent to
$E_1$ or $E_2$ since it is characteristic but $E_i$ is not.
%$\tilde{\phi}(E'_i)=E_i$, $i\leq n-2$ and
%$\tilde{\phi}(E'_{n-1})=H-E_{n-1}-E_n$, one has to apply diffeomorphisms involving
% $E_i$'s for $i\leq n-2$ to send
%$H-E_{n-1}-E_n$ to $E_{n-1}$ or $E_n$.

\rmktwo

\propone \label{Diff action on H2} $D_{K_0}(M)$ is generated by $K_0$-twists.\proptwo

\pfone
%This is actually Corollary \ref{base extension} in the case of $k=n$.
For $\phi\in D_{K_0}(M)$, apply  Corollary \ref{base extension} to $\phi(E_i), 1\leq i\leq n-2$,
 there is a $K_0$-twist $f$ such that
$f(\phi(E_i))=E_i$.

Consider $F_{n-1}=f(\phi(E_{n-1}))$ and $F_n=f(\phi(E_{n}))$. $F_{n-1}$ and $F_n$ are orthogonal to $E_i, 1\leq i\leq n-2,$ since
 $f(\phi(E_j))\cdot E_i=E_j\cdot E_i=0$ for $i\leq n-2$ and $j>n-2$. It is easy to see that the only such classes in $\mathcal E_{K_0}$ are
 $H-E_{n-1}-E_n, E_{n-1}, E_n$. Since $F_{n-1}\cdot F_n=0$, it has to be that $\{F_{n-1}, F_n\}=\{E_{n-1}, E_n\}$.
  By composing $f$ with the $K_0-$twist $R(E_n-E_{n-1})$ if necessary,
one obtains the desired inverse of $\phi$ generated by $K_0$-twists, which means $\phi$ is also generated by $K_{0}$-twists.

\pftwo

We now prove an analogue of Proposition \ref{exceptional sphere classification} for $\LL_{K_0}$.  We start with

\lemmaone \label {b_i>0}
Suppose   $\xi=aH-\sum b_iE_i\in H_2(M;\ZZ)$  is in  $\LL_{K_0}$,
If  $ a>0$ then
$\eta(\xi)=\eta_{K_0}(\xi)$ and
$b_i\geq 0$.
\lemmatwo

\pfone
For any $\xi \in \LL_{K_0}$, $\eta_{K_0}(\xi)=0$ and the minimal genus is $0$ as well. By Lemma \ref{delta}, if $\xi=aH-\sum b_iE_i\in H_2(M;\ZZ)$ with $a>0$,
then $\eta_{K_{\{\delta_i\}}}(\xi)$ is defined.  Recall from the minimal genus assumption and the fact that the symplectic genus is no bigger than
the minimal genus,  $0=\eta_{K_0}(\xi)\geq \eta_{K_{\{\delta_i\}}}(\xi)$ for any
choice of $\{\delta_i\}$. But this holds only if $b_i\geq 0$ for all $i$, hence the conclusion follows.
\pftwo

Following Evans \cite{Evans}, we make the following definition.

\begin{definition}\label{binary-ternary}  A
class is called \textit{binary} if it is of the form $E_i-E_j$, and
\textit{ternary} if it is of the form $H-E_i-E_j-E_k$, $1\leq i,j,k\leq n$.
\end{definition}

Clearly, binary and ternary classes are in  $ \LL_{K_0}$.
In the rest of our paper, we denote $R(H-E_i-E_j-E_k)$ by $\gijk$ for short.

\propone\label{K-Lagrangian sphere class classification} For
$\xi\in\LL_{K_0}$, either $\xi$ is $K_0-$equivalent to a binary or
ternary class. Further, if either  $n\neq 3$, or $n=3$ but $\pm \xi\ne H-E_1-E_2-E_3$, then $\xi$
is $K_0-$equivalent to the  binary class $E_1-E_2$.\proptwo

\pfone Let  $\xi=  aH-\sum b_iE_i$.
When $a=0$ it is easy to conclude that  $\xi$ is binary.
Let $r$ be the number of nonzero $b_i$. An easy calculation
%from (\ref{c_1 condition})(\ref{prelim genus})
verifies the case when $r\leq 3$. Thus  we assume $r>3$
 with $a>0$ by possibly reversing the signs of $\xi$
 (simply do a reflection with respect to $\xi$).
By Lemma \ref{b_i>0}, we may assume that $b_1\geq b_2\geq\cdots\geq b_n\geq0$.

Now we write down the reflection $\gone$ explicitly:

$$\gone(\xi)=(2a-b_1-b_2-b_3)H-\sum c_iE_i,$$

\nono where $c_i=b_i$ for $i>3$.

If $2a-b_1-b_2-b_3<0$,  consider the class $-\gone(\xi)\in \LL_{K_0}$.
In this case,  the leading coefficient of $-\gone(\xi)$ is bigger than $0$.  However, since $r>3$,
one must have $-c_{r}=-b_{r}<0$, a contradiction to Lemma
\ref{b_i>0}. Thus, $2a-b_1-b_2-b_3\geq 0$

Moreover, from Lemma \ref{symplectic genus
0=non-reduced},
 $\xi$ is not reduced
 %but the first
%equation of reduced class is satisfied,
hence one must have
$b_1+b_2+b_3>a$.  Combining these facts, we have

$$0\leq2a-b_1-b_2-b_3<a.$$

Also notice that $\gone(\xi)$ verifies all conditions of Lemma
\ref{b_i>0}, thus $c_i>0$ still holds.  One could then repeat the above process and use induction on the
coefficient $H\cdot\xi$ until $r\leq 3$ or $a=0$.

%Notice that $-id$ commutes with any element in $D(M)$.
\pftwo

\rmkone\label{algorithm} The algorithm reducing a $K$-null spherical
classes is also valid for exceptional classes.  In this case, one
gets an explicit $K_0$-equivalence from an exceptional class to
$E_i$ when $n\geq3$ or possibly $H-E_1-E_2$ when $n=2$.  This is
also used in \cite{MSE}. \rmktwo

%%%%%%%%%%%%%%%%%%%%%%%%%%%%%%%%%%%%%%%%%%%%%%%%%%%%%%%%%%%%%%%%%%%%%%%%%%%%%%%%%%%%%%%%%%%%%%%%%%%%%%%%%%%%%%%

%%%%%%%%%%%%%%%%%%%%%%%%%%%%%%%%%%%%%%%%%%%%%%%%%%%%%%%%%%%%%%%%%%%%%%%%%%%%%%%%%%%%%%%%%%%%%%%%%%%%%%%%

\subsubsection{$(K, \alpha)-$null spherical classes and $D_{K, \alpha}(M)$}

In this section we fix a class $\alpha$ in the $K-$symplectic cone $\mathcal C_K$.

\begin{definition} \label{K-alpha Lag}
 A \textit{$(K, \alpha)-$null spherical class} is a $K-$null spherical class which
pairs trivially with $\alpha$.
\end{definition}

Reflections $R(\xi)$, for $\xi$ a $(K, \alpha)-$null spherical class, are called \textit{$(K, \alpha)-$twists}.
We also define  $D_{K, \alpha}(M)$ to be the subgroup of $D_K(M)$ preserving $\alpha$. One has the following easy observation:

\lemmaone\label{phi}
If $\phi\in D_K$ then
\begin{itemize}
\item  $\phi$ induces a bijection from $\LL_{K, \alpha}$ to $\LL_{K, \phi^{-1}(\alpha)}$.
\item
$f\to \phi^{-1}\circ f\circ \phi$ defines an isomorphism from $D_{K, \alpha}$ to $D_{K, \phi^{-1}(\alpha)}$
taking $R(\xi)$ to $R(\phi(\xi))$.
\item $\alpha$ has a positive lower bound on $\EE_K$ which is attained by some $K$-exceptional class.

\end{itemize}
\lemmatwo
%%%%%%%%%%%%%%%%%%%%%%%%%%%%%%%%%%%%%%%%%%%%%%%%%%%%%%%%%%%%%

%%%%%%%%%%%%%%%%%%%%%%%%%%%%%%%%%%%%%%%%%%%%%%%%%%%%%%%%%%%%%%%%%%%%%%%%%%%%%%%%%%%%

\nono The third assertion is a consequence of Gromov compactness and the well-known fact that, for any
$E\in\EE_K$, $GT(E)\neq0$ with respect to any symplectic form $\omega$ representing $\alpha$.  We are now ready to prove
the following:

\propone\label{homological twists}
$D_{(K_0, \alpha)}$ is generated by $(K_0, \alpha)-$twists.
\proptwo

\pfone We will use induction on $n$ and a trick due to Martin Pinsonnault \cite{Pinsonnault}.
For $n\leq 3$ this is easy to verify directly by listing all exceptional classes.

If $n\geq 3$ choose $\{E_i'\}_{i=1}^{n-2}\subset \EE_{K_0}$ such that  $E_1'$ has
minimal $\alpha$-area, and  $E_i'$  has minimal
$\alpha$-area among exceptional classes orthogonal to $E_j$ for all $j<i$.
By  Corollary \ref{base
extension}, there is $\psi\in D_{K_0}(M)$ such that $\psi(E_i')=E_i$.
%Note that one does not have a choice when $i=n-1$ or $n$.
By Lemma \ref{phi} we can assume that $E_i'=E_i$, so that among the basis elements
$\{H,E_1,\cdots ,E_n\}$, $E_1,\cdots, E_{n-2}$ enjoys the above minimality property.

Let $f\in D_{(K_0, \alpha)}$.  If
one could find a series of $(K_0, \alpha)-$twists such that their composition
$\phi$ satisfies $\phi\circ f (E_1)=E_1$, one can then include
$\phi^{-1}$ into our decomposition of $f$.  Since $E_1$ is orthogonal to $\phi\circ f(E_i)$ for
$i\neq 1$, one can then use induction on these classes.  Therefore it suffices to look
 for such a $\phi$ in the rest of the proof.

Notice first that
\begin{equation}\label{ijk}
\alpha(H-E_i-E_j-E_k)\geq0, \hskip2mm i>j>k.
\end{equation}

\nono This is clear from the construction: since the $K_0$-exceptional class $(H-E_i-E_j)\cdot
E_l=0$, for all $l<k$ and $k\leq n-2$, we have $\alpha(H-E_i-E_j)\geq\alpha(E_k)$

%Recall from the proof of Proposition \ref{exceptional sphere
%classification}, one can always reduce an exceptional class to the
%form of $H-E_i-E_j$ using a series of Cremona transforms.

Assume
$f(E_1)=aH-\sum b_{r_i}E_{r_i}$. Notice that $f(E_1)\in \EE_{K_0}$ and $\alpha(f(E_1))=\alpha(E_1)$. If $a=0$ then
  $f(E_1)=E_k$
for some $k$ and
$E_1-E_k \in \LL_{K_0, \alpha}$. In particular, $R(E_1-E_k)\in D_{K_0, \alpha}$ and we can choose $\phi=R(E_1-E_k)$.

If $a\ne 0$, by Lemma \ref{-1}, $a>0$ and $b_i\geq 0$.
 Suppose  $b_{r_1}\geq
b_{r_2}\geq\cdots\geq b_{r_n}\geq 0$.   Now apply $\Gamma_{r_1r_2r_3}$,
$$\Gamma_{r_1r_2r_3}(f(E_1))=f(E_1)+(a-b_{r_1}-b_{r_2}-b_{r_3})(H-E_{r_1}-E_{r_2}-E_{r_3})$$
\nono From Lemma \ref{symplectic genus 0=non-reduced},
$a-b_{r_1}-b_{r_2}-b_{r_3}<0$. By \eqref{ijk}, $\alpha(H-E_{r_1}-E_{r_2}-E_{r_3})\geq 0$, thus

$$\alpha(E_1)=\alpha(f(E_1))\geq\alpha(\Gamma_{r_1r_2r_3}(f(E_1))).$$
By the choice of $E_1$, we must have $\alpha(H-E_{r_1}-E_{r_2}-E_{r_3})= 0$. This means that
$H-E_{r_1}-E_{r_2}-E_{r_3}\in \LL_{K_0, \alpha}$ and $\Gamma_{r_1r_2r_3}\in D_{K_0, \alpha}(M)$.

Now from Remark \ref{algorithm}, by repeating the above operations
we eventually have an equivalence between $E_1$ and $E_k$ for some
$k$.
%$H-E_{i_0}-E_{j_0}$ for some $i_0$ and $j_0$.
 Denote their
composition to be $\tilde{\phi}$.

If $k=1$ we let $\phi=\tilde \phi$. If $k\ne 1$, then $\alpha(E_k)=\alpha(E_1)$ and
we let $\phi=R(E_1-E_k)\circ \tilde \phi$.

\pftwo
\subsection{Irrational ruled manifolds}
 It is clear that a
minimal symplectic irrational ruled manifold does not admit any Lagrangian spheres. Thus,
in this subsection, $M=(\Sigma_h\times S^2)\# n\overline{\CC P}^2$.
Any non-minimal genus $h$ ruled manifold is of this form.
Define $\EE, \LL, \KK, D(M)$ as above. For $K\in \KK$ also define $D_K(M), \EE_K, \LL_K$ and $K-$null spherical class as above.

A standard homology basis consists of $\{T,F,E_1,\cdots,E_n\}$,
with the following algebraic properties:
\begin{equation} \label{sb} T\cdot F=1, \quad  T^2=F^2=T\cdot E_i=F\cdot E_i=0,  \quad E_i^2=-1, \quad 1\leq i\leq n.
\end{equation}
Geometrically, $T$
is  represented by a surface with
genus $h$, $F$ the class of a fiber, and  $\{E_i\}$ a maximal collection of orthogonal  exceptional classes in $\EE$.
The standard canonical class is then $K_0=PD(-2T+(2h-2)F+\sum E_i)$.

$D(M)$ is characterized as the subgroup of $Aut (H_2(M;\ZZ))$ preserving $F$ up to sign (\cite{FM1}).
Due to  the transitive action of $D(M)$ on
$\mathcal{K}$ shown in \cite{LL}, we may
again restrict to the case $K_{\omega}=K_0$.

\begin{lemma}\label{rule1}
$\EE_{K_0}=\{E_i, F-E_i, i=1,..., n\}$.

$\LL_{K_0}=\{\pm (F-E_i-E_j), \pm (E_i-E_j), 1\leq i<j \leq  n\}$.

% Furthermore, when $n\geq 2$, $D_{K_0}$ acts transitively on $\EE_{K_0}$ and $\LL_{K_0}$.
\end{lemma}

\begin{proof}
First of all, if  $\xi=aT+bF+\sum c_iE_i$ is represented by a
sphere, then  $a=0$. This follows from the fact that a
sphere does not have a nonzero degree map to  a positive genus
curve.

With this understood, it is easy to determine $\EE_{K_0}$ and $\LL_{K_0}$ using \eqref{sb}.

\end{proof}

For $\alpha\in \CCC_{K_0}$ we define $D_{K_0, \alpha}$, $(K_0, \alpha)-$twist as before.

\begin{lemma}\label{rule2}
$D_{K_0, \alpha}$ is generated by $(K_0, \alpha)$ twists.
\end{lemma}

\pfone
As in the rational manifold case,  we  do induction on $n=b^-(M)+1$.

When $n=1$, since $\phi(F)=\pm F$, it is easier to see that $D_{K_0}$, and hence $D_{K_0, \alpha}$, is trivial.

In general when $n\geq 2$, for $\phi \in D_{K_0, \alpha}$ we consider its action on $\EE_{K_0}$.
Let $E$ be the exceptional class with minimal
$\alpha$ area, the induction is immediate if
$\phi(E)\cdot E=0$, in which case we simply compose $\phi$ with  the $(K_0, \alpha)-$
twist  $R(E-\phi(E))$ to reduce to a lower $n$ case.

Otherwise,   $\phi(E)=F-E$ by Lemma \ref{rule1}.  In this case $2\alpha(E)=\omega(F)$.
Since two classes $A$ and $F-A$ are either both
in $\EE_{K_0}$ or neither, the minimality of $\alpha(E)$ forces all other exceptional spheres
to have the same area as $E$.   Since $n\geq 2$, it is clear that one
could send $F-E$ back to $E$ via a composition of $(K_0, \alpha)-$twists, for example,  the $(K_0, \alpha)-$ twist $R(E'-E)$
followed by  $R(F-E'-E)$,  where $E'$ is another
exceptional standard basis element orthogonal to $E$.
Again we are able to reduce to a lower $n$ case.

\pftwo

\section{Lagrangian spherical classes when $b^+=1$}\label{section:Lagrangian spherical}

Theorem \ref{general} allows us to effectively apply a
Lagrangian-relative version of inflation procedure in this section.
Together with
Proposition
\ref{K-Lagrangian sphere class classification}, this enables us to
classify
 Lagrangian spherical classes
 in symplectic  4-manifolds with $\kappa=-\infty$.
We also give the proof of Theorem \ref{homological action theorem} in  \ref{subsection:homological action}.

\subsection{Lagrangian relative inflation}

The inflation procedure was first introduced by Lalonde
\cite{Lalonde} and proved useful in many fundamental problems in
symplectic geometry (see \cite{LM2} for example).

The inflation construction in \cite{Lalonde},  together with Theorem \ref{general},
gives

\lemmaone[{\bf Inflation Lemma}] \label{inflation}
Let $L$ be a Lagrangian sphere in a symplectic 4-manifold with $b^+=1$.
Let $A$ be a class in $H_2(M;\ZZ)$ satisfying the condition in Theorem \ref{general}.   Assume also that $A\cdot [L]=0$.
 Then  there is a
 closed form $\rho$  on $M$ in class $PD(A)$ supported away from $L$ so that

$$\beta_t=\omega+t\rho, \hskip 3mm t\geq 0,$$

\nono is symplectic. In particular, $L$ remains Lagrangian for any $\beta_t$.

\lemmatwo

The proof is straightforward: note in \cite{Lalonde}, $\rho$ is supported near a symplectic surface in
class $A$.  Therefore, if such a symplectic surface is disjoint from
the given Lagrangian sphere $L$, $L$ remains Lagrangian in the course of
the inflation procedure. Now Theorem \ref{general} provides the desired
symplectic surface.

We first apply Lemma \ref{inflation} to study  symplectic ball
embeddings in the complement of a Lagrangian sphere.
P. Biran and O. Cornea
 studied Lagrangian relative  embeddings in  \cite{BC} (called \textit{mixed
packing} there),
where the size of maximal ball embeddings is found in some cases.

In our case of a Lagrangian sphere $L$ in a symplectic 4-manifold with $b^+=1$,  Lemma \ref{inflation} enables us to show that    packing problems in the complement of
$L$  can often be answered in the same way as for the ordinary packing problems.
Here is one example.
Biran showed in \cite{BiranS} that in any closed symplectic $4$-manifold
with an integral symplectic form,
the symplectic packing problem is stable via inflation on a Donaldson hypersurface.
For a symplectic 4-manifold $(M, \omega)$
with $b^+=1$ and  $\omega$ integral, the class $n[\omega]$ for $n$ large  satisfies the conditions in Theorem \ref{general} for
an arbitrary given Lagrangian sphere.
Thus Lemma \ref{inflation} can  be applied to such a class and hence Biran's stability result is also valid for $M\backslash L$.

\rmkone\label{New Connectedness Theorem}
It would be useful to prove
the following parameterized version of Lemma \ref{inflation}, which would be the analogue of  Lemma 1.1 in
\cite{MIsotopy}:
%\lemmaone[{\bf Inflation Lemma}] \label{inflation'}
Given a path $\omega_t$, $0\leq t\leq 1$, of symplectic forms on $M$ with $b^+=1$ and
a sphere $L$ Lagrangian for  each $\omega_t$.
Let $A$ be a class in $H_2(M;\ZZ)$ satisfying the conditions in Theorem \ref{general}.   Assume also that $A\cdot [L]=0$.
Then there is a path
$\rho_t$ of closed forms on $M$ in class $PD(A)$ supported away from $L$ so that

$$\beta_t=\omega_t+\kappa(t)\rho_t,\hskip 3mm 0\leq t\leq 1,$$

\nono is symplectic whenever $\kappa(t)\geq 0$. In particular, $L$ remains Lagrangian for any $\beta_t$.

Lemma 1.1 in \cite{MIsotopy} is used to show that the ball embedding space
$$E_{\bar{\lambda},k}(M,\omega)=\{\psi| \psi:\coprod_{i=1}^k(B_4(\lambda_i),\omega_{std})\hookrightarrow (M,\omega)\}$$
%, \quad Im (\psi)\bigcap L=\emptyset\}$$
with $\bar{\lambda}=(\lambda_1,\ldots,\lambda_k)$, is connected.
Substituting   Lemma 1.1 in \cite{MIsotopy} by its $L$ relative version as above  in appropriate places, we would be able to    obtain
the connectedness of the relative ball embedding space.

\rmktwo

\subsection{Existence of Lagrangian spheres}\label{section:existence}
In this subsection we present the proof of Theorem \ref{Lagrangian
sphere class classification}. We begin with some general discussions of
Lagrangian spheres in  a non-minimal symplectic 4-manifold with $b^+=1$.

\subsubsection{Non-minimal 4-manifolds with $b^+=1$ and $\kappa\geq 0$}\label{existence in other manifolds}

We begin with the following two persistence results.

\lemmaone\label{1-blowup}
Let $( M,\omega)$ be a symplectic 4-manifold with
$b^+(M)=1$, $[\omega]\in H^2(M;\QQ)$. Let $(\overline M, \overline \omega)$ be the  one point blow up of $(M, \omega)$  with size $a$,
and $\iota:H_2(M;\ZZ)\to H_2(\overline M, \ZZ)$  the canonical injection.
If $L\subset (M, \omega)$ is a Lagrangian sphere, then there is a Lagrangian sphere in $(\overline M, \overline \omega)$ in the class $\iota([L])$.

\lemmatwo

\pfone By the uniqueness of blow ups (Corollary 1.3 in \cite{MIsotopy}), we can place the ball of size $a$ anywhere in $(M, \omega)$. If the ball is
  disjoint from $L$, we are done. Otherwise, first choose a  ball of size $a'<a$ and disjoint from $L$, we obtain a  blow up $(\overline M, \overline \omega')$ with a Lagrangian $\bar L$ from $L$.
Let  $p:\overline M\rightarrow M$ be a topological blow down map which contracts the exceptional sphere.   Consider the class $\beta_{l, \delta}= l([p^*\omega]-(a+\delta)PD(E))$ for $\delta>0$.
Clearly, $\beta_{l, \delta}([\bar L])=0$.
Since the $K_{\overline \omega}-$symplectic cone $\CCC_{K_{\overline \omega}}$ is open, we can  assume that $\beta_{l, \delta}$ is in
 $\CCC_{K_{\overline \omega}}$ by choosing $\delta$ small.
If  $a+\delta$ is further assumed to be a rational number, then there exists $l\in\ZZ^+$ such that $\beta_{l, \delta}$ satisfies the conditions in Lemma \ref{inflation}.
Applying  Lemma \ref{inflation} to such a $\beta_{l, \delta}$ and $\bar L$,
we find that $\bar L$ remains Lagrangian in $(\overline M, \overline \omega'')$, where $\overline \omega''$ is a symplectic form  in the class $[p^*\omega]-aPD(E)$ up to a rescale.
The proof is finished by again invoking the uniqueness of blow ups.
\pftwo

If $E$ is the class of the exceptional sphere, this lemma can be  viewed as the persistence of Lagrangian spheres under a symplectic deformation on $\overline M$ in the $E$ direction, which can also be proved via the inflation construction along a symplectic surface with negative self intersection
 as in \cite {LU}.

\lemmaone \label{persistence} Let $(\overline M,\overline \omega)$ be a symplectic 4-manifold with
$b^+(\overline M)=1$, $[\overline \omega]\in H^2(\overline M;\QQ)$.  If there are two orthogonal exceptional classes $E_1$, $E_2\in \EE_{\omega}$  with
equal symplectic area $a$, then there is a Lagrangian sphere in the binary class $E_1-E_2$.
%Assume also that $PD([\omega])-aE_1-aE_2\in H_2(M;\mathbb{Q})$.
 \lemmatwo

\pfone
Let us first consider a local model: the two point blow up of  a standard ball  with equal size $t>0$. This can be identified with
the complement of a line in $\CC P^2\# 2\overline{\CC P}^2$ with a symplectic form $\tau$  with $[\tau]=PD(H-t E_1 -t E_2)$.
%Without loss of generality, we can assume $\tau(H)=1$,
  Notice that $(\CC P^2\# 2\overline{\CC P}^2, \tau)$ is symplectomorphic to
a one point blow up of a monotone $S^2\times S^2$ with size $1-2t$.  If we apply
Lemma \ref{1-blowup} to the antidiagonal $L_a$ in this monotone $S^2\times S^2$,
we find a Lagrangian sphere in $(\CC P^2\# 2\overline{\CC P}^2, \tau)$ in the class $E_1-E_2=\iota([L_a])$.  In addition, such a Lagrangian sphere can be made disjoint from an embedded
$H$-class sphere in $\CC P^2\# 2\overline{\CC P}^2$ by Theorem \ref{main}.  We therefore obtain a Lagrangian sphere in our local model.

In general, let $(M, \omega)$ be obtained by symplectically blowing down two disjoint spheres in $E_1$ and $E_2$ in $(\overline M,\bar \omega)$ and adopt notations
in Lemma \ref{1-blowup}.
We shrink both balls corresponding to $E_1$ and $E_2$ to size $\epsilon\ll 1$.  By the uniqueness of ball-embeddings (in case of absence of a Lagrangian sphere, see Remark \ref{New Connectedness Theorem}),
we may place the two tiny balls $V_1$ and $V_2$ in a Darboux chart.  Our local model analysis above ensures that there is a Lagrangian sphere $L$ in the blow-up of the chart around $V_1$ and $V_2$.
Consider the class $B_b=PD(p^*\omega)-bE_1-bE_2$ where $b$ is a positive  rational number slightly larger than $a=\overline \omega(E_i)$, $i=1,2$.
Since the $K_{\overline \omega}-$symplectic cone $\CCC_{K_{\overline \omega}}$ is open, we can further assume that $PD(B_b)$ is in
 $\CCC_{K_{\overline \omega}}$.  Clearly,  $B_b\cdot (E_1-E_2)=0$.
Thus for some large integer $l_b$, $l_bB_b$ satisfies the conditions in Lemma \ref{inflation}.  Now the conclusion follows
from inflating along a symplectic surface in class $B_b$ as in the proof of Lemma \ref{1-blowup}.

\pftwo

\corone\label{b+=1 existence} Suppose $(M,\omega)$ is a minimal
symplectic manifold with $b^+=1$, $[\omega]\in H^2(M,\QQ)$. Suppose
$(\overline M,\bar{\omega})$ is a $k$ point  symplectic blow-up of $(M, \omega)$ with $E_i, i=1,..., k,$ the corresponding exceptional class, and the
canonical injective map is denoted as: $\iota:H_2(M;\ZZ)\to
H_2(\overline M;\ZZ)$. Then $\xi\in H_2(\overline M;\ZZ)$ is a Lagrangian
spherical class if
\begin{enumerate}[(1)]
\item either $\xi\in Im(\iota)$ and $\iota^{-1}(\xi)$ is Lagrangian
spherical,
\item or $\xi=E_i-E_j$ for some $i, j$, i.e. $\xi$ is binary, and $\omega(\xi)=0$.
\end{enumerate}

If  $\kappa(M)\geq 0$, these are the only Lagrangian spherical classes of $(\overline M, \overline \omega)$.
\cortwo

\pfone (1) and (2)  follow directly from Lemmas \ref{1-blowup} and \ref{persistence} respectively.

To show these are the only Lagrangian spherical classes when $\kappa(M)\geq 0$,
suppose $\xi=\xi'-\sum_{i=1}^k a_iE_i$ is
represented by a Lagrangian sphere $\bar L$, where $\xi'\in Im(\iota)$.
%$E_i$'s the exceptional classes.

If $a_i=0$ for all $i$, then apply Theorem \ref{main} to find disjoint exceptional spheres in the classes $E_i$, which are also disjoint from
$\bar L$. This shows that $\xi'$ is a Lagrangian spherical class of $(M, \omega)$.

Now assume  some $a_i\neq0$.
The reflection
$R(\xi)$ thus sends $E_1$ to $a\xi'-\sum_{i>1} a_iE_i-(a_1^2-1)E_1$.
Such a class is an exceptional class of $(\overline M, \bar \omega)$.  However, from
the uniqueness of the minimal model for symplectic manifolds with $\kappa\geq 0$ (\cite{Mcduff}), $a\xi'-\sum_{i>1}
a_iE_i-(a_1^2-1)E_1=E_j$ for some $j$.  This shows $\xi'=0$  and $\xi$ is indeed binary.

\pftwo

\subsubsection{Rational manifolds}

\pfone[Proof of Theorem \ref{Lagrangian sphere class
classification}, rational manifold case:] The case of $S^2\times
S^2$ is well-known and so we focus on blow-ups of $\CC P^2$ below.

Due to  the transitive action of $D(M)$ on
$\mathcal{K}$ mentioned in Section \ref{section:K-Lagrangian},
and using definition \ref{K-alpha Lag}, we are reduced to prove the following Proposition.

\begin{prop}
Suppose $M=\CC P^2\# n \overline{\CC P}^2$ with $\{H, E_1, \cdots, E_n\}$  a standard basis, and $\omega$ is a symplectic form
with $K_{\omega}=K_0=PD(-3H+E_1+\cdots+ E_n)$. Then $\xi\in H_2(M;\ZZ)$ is represented by a Lagrangian sphere if and only if
$\xi$ is $(K_0, [\omega])-$null spherical.
\end{prop}

\pfone
The
conditions are clearly necessary. In the
case $n=2$, up to sign, the only $K_0-$null spherical class is the binary class $\xi=E_1-E_2$.
And if $\xi$ is $(K_0, [\omega])-$null spherical, then $E_1$ and $E_2$ must have equal symplectic area.
 Thus the existence of a Lagrangian sphere  has been  argued  in the first paragraph of Lemma \ref{persistence}.

Let us then suppose  that $n>3$.
One notices that in this case
$\xi$ can also be assumed to be binary. This is because, from Proposition
\ref{Diff action on H2},  there is a self-diffeomorphism $\phi$ of $M$, which induces a $K_0$-twist on homology and
sends $\xi$ to a binary class,
 and we could just consider $\phi_*(\xi)$ in $(M,(\phi^{-1})^*\omega)$.  Without loss of generality
we could further assume $\xi=E_1-E_2$. If $\omega(\xi)=0$, then, up to scaling,  $PD([\omega])=3H-\sum b_iE_i$ with $b_1=b_2=b>0$.

Blowing down a collection of disjoint exceptional spheres in the classes $E_i$ with $i\geq 3$, we obtain
$M'=\CC P^2\# 2\overline{\CC P}^2$ with a symplectic form $\omega'$ in the class $[\omega]=PD(3H-bE_1-bE_2)$.
As just shown, there is a Lagrangian sphere $L\subset (M', \omega')$ in the class $E_1-E_2$.
Now apply Lemma \ref{1-blowup} to obtain the desired  Lagrangian sphere back in $(M, \omega)$ by performing $n-2$ blow-ups.

Finally let us suppose that  $n=3$.  A $K_0-$null spherical class is either binary or the ternary class $\xi=H-E_1-E_2-E_3$.
The binary case can be treated in the same way as in the case $n>3$.  So let us assume that
$\xi=H-E_1-E_2-E_3$.  Let $(\bar M, \bar \omega)$ be a one point  blow up of $(M, \omega)$, $E_4$ the new exceptional class, and $\iota$ the canonical map. Notice that $b^-(\bar M)=4$ and $\iota(\xi)$
is $(K_0, [\bar \omega])-$null spherical, thus there is a Lagrangian $\bar L\subset (\bar M, \bar \omega)$ in the class $\iota(\xi)$. By applying Theorem \ref{general} to $\bar L$ and $E_4$,
we conclude the proof by blowing down an exceptional sphere in class $E_4$ disjoint from $\bar L$.

\begin{comment}
Consider a diffeomorphism $\phi$ inducing the spherical reflection  $\Gamma_{124}$.
Since $\phi_*(\xi)=E_4-E_3$, there is a Lagrangian sphere $L'$ in the binary class. By
performing again the same Dehn twist, one gets back to $M'$ with a
Lagrangian sphere $L$ in class $\xi$.  Theorem \ref{general} then
applies to $L$ and the exceptional class $E_4$.  Blowing down an
exceptional sphere in $E_4$ thus concludes the proof.
\end{comment}

\pftwo
Now the proof of Theorem \ref{Lagrangian sphere class classification} in the rational manifold case is complete.
\pftwo
\subsubsection{Irrational ruled manifolds}\label{section:irrational ruled-existence}
\pfone[Proof of Theorem \ref{Lagrangian sphere class
classification}, irrational ruled manifold case:]

Similar to the rational case, it reduces to the following statement.
\pftwo
\begin{prop}
Suppose $M=(\Sigma_h\times S^2)\# n \overline{\CC P}^2$ with $\{T,F, E_1, \cdots, E_n\}$  a standard basis, and $\omega$ is a symplectic form
with $K_{\omega}=K_0=PD(-2T(2h-2)F+E_1+\cdots+ E_n)$. Then $\xi\in H_2(M;\ZZ)$ is represented by a Lagrangian sphere if and only if
$\xi$ is $(K_0, [\omega])-$null spherical.
\end{prop}

\pfone
We use the cut and paste procedure in \cite{Mcduff} to reduce it to the rational manifold case.

We can view $(M, \omega)$ as a symplectic genus 0 Lefschetz fibration over $\Sigma_h$ with
$n$ reducible fibers, each consisting of a pair of exceptional spheres in the classes $E_i$ and $F-E_i$.
Denote the projection by $\pi$ and the image
of the reducible fibers by
  $B$.
 View  $\Sigma_h$
as assembled from a $4h$-sided polygon  with  the vertices going to
 $x_0\in \Sigma_h$, the edges going to a $2h-$wedge of loops $\Lambda_h$.
 Since $B$ is a finite set, we can assume that $B\cap \Lambda_h=\emptyset$.

We  cut $M$ along  $\pi^{-1}(\Lambda_h)$ to
obtain a genus 0  Lefschetz fibration  $V$ over a two disk $D$ with $n$ reducible fibers.
Recall from Lemmas 4.13 and  4.14 in
\cite{Mcduff}  that with a symplectic deformation
supported near an arbitrarily small neighborhood of $x_0$,
$(M, \omega)$ can be assumed to be a symplectic product in a neighborhood of $\pi^{-1}(\Lambda_h)$.  Therefore we can compactify  $(V, \omega)$ into  a genus 0 Lefschetz fibration $(\bar V, \bar \omega)$ over $S^2$ with
$n$ reducible fibers
by adding a fiber $F_0$.

Notice that $V$ is diffeomorphic to
$(S^2\times D^2)\# n\overline{\CC P}^2 $, and $\bar V$ is diffeomorphic to
$(S^2\times S^2)\# n\overline{\CC P}^2 $=$(\CC P^2 \# \overline{\CC P}^2)\# n\overline{\CC P}^2 $.
Moreover, in the standard basis representation, $F$ corresponds to $H-E_1$, and $E_i$ corresponds to $E_i$.
In particular, a $(K_0, [\omega])-$null spherical class corresponds to either $H-E_1-E_i-E_j$ or $E_i-E_j$, $2\leq i<j\leq n$.

We have shown there are Lagrangian spheres in $(\bar V, \bar \omega)$ in these classes.
What remains to prove is that there are  Lagrangian spheres  disjoint from the symplectic sphere $F_0$.
This is true due to Theorem \ref{main}, since $[F_0]=H-E_1$ is a square 0 class,
orthogonal to  $H-E_1-E_i-E_j$ and  $E_i-E_j$ for any  $2\leq i<j\leq n$.

%Moreover, by considering the fibration as
%a family of pseudo-holomorphic curves as in \cite{Mcduff}.  By deleting
%a codimension $-2$ subset of $M$ containing spheres in exceptional classes,
% one still obtains a smooth fibration over
%$\Sigma\backslash\Lambda$, where $\Lambda\subset\Sigma_h$ is a finite set.
%Thus, one could
%always arrange such cut-and-paste operation above to be supported
%away from  exceptional spheres in each $E_i$.

%Now after the above operation, we turned an irrational ruled
%manifold into a rational ruled manifold with Euler number at
%least $6$,
%with the class $F-E_i-E_j$ turned into a binary class.
%Our existence result follows then from the rational manifold case.  The only
%subtlety is to avoid a given small neighborhood of $F_0$ (which corresponds to where the deformation takes place).

%From
%the proof in the rational case with $2$ blow-ups, we indeed use a deformation
%supported near the exceptional spheres involved in the binary class,
%then construct a local Lagrangian sphere then use the inflation.
%Therefore, as long as the local Lagrangian sphere does not touch the
%given neighborhood, which is easily achieved, then our proof goes
%through. This concludes our proof.

\pftwo

\subsection{Homological action}
%Proof of Theorem \ref{homological action theorem}}
\label{subsection:homological
action}

We are now ready to prove Theorem
\ref{homological action theorem}.

\pfone Let $(M,\omega)$ be
 a symplectic  4-manifold with $\kappa=-\infty$. Further assume that a standard basis is chosen.
As mentioned in the proof of Theorem \ref{Lagrangian sphere class classification},
 fixing the canonical class causes no loss of generality.  Thus we assume that $K_{\omega}=K_0$.

On the one hand, if $f\in Symp(M, \omega)$, then $f_*\in D_{K_0, [\omega]}(M)$.
On the other hand,  Theorem \ref{Lagrangian sphere class classification} implies any $(K_0,[\omega])$-twist
 is realized by a Lagrangian Dehn twist.  With this understood,
Theorem \ref{homological action theorem} is simply a consequence of Proposition \ref{homological twists}, Lemma \ref{rule2},  and Theorem \ref{Lagrangian sphere class classification}.

\pftwo

%\nono In particular,
%in the rational manifold case when the blow-ups are of
%equal size, we clearly have the following corollary:
\corone If $(M,\omega)$ is monotone, the representation of
the symplectic mapping class group on $H_2(M;\ZZ)$, namely, the Torelli part,
is $D_{K_\omega}(M)$.
\cortwo

%\corone
% If the blow-ups are of
%equal size but $(M, \omega)$ is not monotone, the action is generated by
%$R(E_i-E_j)$ and is identical to the symmetric group $S_n$ permuting
%the exceptional spheres.\cortwo

\rmkone Corollary \ref{b+=1 existence} also has its counterpart, which asserts that, when $b^+(M)=1$ and $\kappa(M)\geq0$,  the
homological action of  $Symp(\overline M, \bar \omega)$ is generated by the homological action of
$Symp(M, \omega)$ and binary Lagrangian reflections.

%This follows directly from
%the uniqueness of minimal models when $\kappa(M)\geq 0$.  Combining this with the
%uniqueness  of blow-ups, it is not hard to see that the
%subgroup of homological actions of $Symp(\overline M, \bar \omega)$ is indeed the
%product of actions of $Symp(M, \omega)$ and Dehn twists coming from binary
%spheres in this case.
%Notice in contrast that, for all symplectic
%manifolds with $b^+=1$,  the part of homological action fixing the
%homology of its minimal model
% is generated by Lagrangian Dehn twists.
It would be  interesting to know  whether for any
minimal $(M, \omega)$ with $b^+=1$ and $\kappa(M)\geq 0$,  the homological action of $Symp(M, \omega)$  is generated  by Lagrangian Dehn twists.

%We summarize our discussion as follows:

\rmktwo

%\corone
% Suppose $(M, \omega)$ is a symplectic 4-manifold with $b^+=1$ and $\kappa\geq 0$. Then
% the the part of homological action fixing the homology of its minimal model
% is generated by Lagrangian Dehn twists in the classes of $E_i-E_j$ with $\omega(E_i)=\omega(E_j)$.
% \cortwo

\section{Uniqueness of Lagrangian spheres in rational manifolds}\label{section: Lagrangian uniqueness}

The present section is devoted to the proof of Theorem
\ref{Lagrangian isotopy uniqueness}.  We begin by reviewing two
basic uniqueness results of Hind for $S^2\times S^2$ and $T^*S^2$.
%We offer
%in the hope  of providing some new perspectives.

%%%%%%%%%%%%%%%%%%%%%%%%%%%%%%%%%%%%%%%%%%%%%%%%%%%%%%%%%%%%%%%%%%%%5

%%%%%%%%%%%%%%%%%%%%%%%%%%%%%%%%%%%%%%%%%%%%%%%%%%%%%%%%%%%%%%%%%%%%%%%%%%%%%%%%%%%%%%%%%%%%%%

%\subsection{Uniqueness up to symplectomorphism}\label{symplectic uniqueness in rational}
\subsection{Review of Hind's results}

\subsubsection{$S^2\times S^2$ via symplectic cut}\label{section:alternative approach-cut}
For $S^2\times S^2$ we have the  uniqueness up to isotopy  in
\cite{Hind}:

\thmone[Hind]\label{Hind's theorem S2 times S2}  Lagrangian spheres in a monotone $S^2\times S^2$ are unique up to Hamiltonian isotopy.
\thmtwo

From the connectedness of $Symp(S^2\times S^2,
\sigma\oplus\sigma)$ by Gromov \cite{Gromov}, Theorem
\ref{Hind's theorem S2 times S2} is equivalent to
%shows $L_1$ and $L_2$ are Hamiltonian isotopic, which is the original form that
%  Hind proved

\begin{prop}\label{Hind's} Lagrangian spheres in a monotone
$S^2\times S^2$  are unique up to
symplectomorphisms.
\end{prop}

We here offer an argument for this weaker version of uniqueness  using
 an idea from Hind \cite{Hind2}  turning the Lagrangian uniqueness problem
into a symplectic uniqueness  problem via symplectic cut.
Such an argument is useful for the uniqueness of  Lagrangian $\RR P^2$ in rational manifolds (see \ref{rp2}).
Some preparations are in order.

Denote by $A$, $B\in H_2(S^2\times S^2;\ZZ)$ the classes of two product factors
on $S^2\times S^2$. Let $\Omega_\lambda$  be the product symplectic
form $\pi_1^*\sigma+(1+\lambda)\pi_2^*\sigma$ with $\lambda>0$.
%Due to
%Lalonde-McDuff's theorem, $\omega_\lambda$ is unique up to symplectomorphisms.
Let $\JJ_{\lambda}$ be the
space of $\Omega_\lambda$-tamed almost complex structures.
The following is due to Abreu and
McDuff:

\thmone[\cite{AM}, Proposition 2.1, Corollary 2.8]\label{AbM}
Suppose  $l-1<\lambda\leq l$, $l$ an integer. Then $\JJ_\lambda$ admits a stratification $\{U_k\}_{0\le k\le l}$ with the following properties:

\begin{enumerate}[(1)]
\item For any $J\in U_k$, the class $A-kB$ is represented by a
unique embedded $J$-holomorphic sphere;
\item Each $U_k$ is connected.
\end{enumerate}

\thmtwo

 As a consequence, we have the following claim:

\propone\label{negative connectedness theorem} The space of
symplectic spheres with self-intersection $-2k$ in  $(S^2\times S^2, \omega_{\lambda})$ is
non-empty and connected if $\lambda>k-1$. \proptwo

\pfone A symplectic sphere with self-intersection $-2k$ is in the class $A-kB$, and it
 exists if and only if  $\lambda>k-1$. For two such
symplectic spheres $C_i, i=0,1$, there are
almost complex structures $J_i\in U_k$ such that
$C_i$ is $J_i$-holomorphic for $i=0,1$. By Theorem \ref{AbM} (2), there is a
path $J_t$  in $U_k$ connecting $J_0$ and $J_1$. By Theorem \ref{AbM} (1), there is a unique sphere $C_t$  with self-intersection $-2k$ for each $J_t$.  This path of symplectic spheres is
continuous due to  Gromov's compactness.
%$C_{t_i}$ converges to
%a nodal $J_{t_0}$-curve. However, $J_{t_0}$ does not admit
%holomorphic curve with self-intersection less than $(-2k)$,
%therefore, the limit $J_{t_0}$-curve must be a smooth $(-2k)$-curve.
%This verifies our claim.

\pftwo

%Before giving the actual proof, it may be helpful to sketch the idea.
%Given two genus-$0$ Lagrangian $L_1$ and $L_2$ in a rational or ruled symplectic manifold $M$,
% one performs symplectic cuts near symplectomorphic neighborhoods of $L_i$'s to
%get two new symplectic manifolds $M_i$, $i=1,2$ along with a
%symplectic $(-2)$-sphere each as the cut locus.
%With the help of Lalonde-McDuff's theorem in \cite{LM1}\cite{LM2},
%which is generalized later to non-minimal case in \cite{LL}
%asserting $M_i$'s are symplectomorphic, one only need to check that
%the the space of symplectic $(-2)$-spheres is connected.

%Along the way we try to understand how the Lagrangian isotopy problem goes back and forth between
% the compact and non-compact cases.

\pfone[Proof of Proposition \ref{Hind's}:] Given two Lagrangian
spheres $L_1, L_2$ in $S^2\times S^2$ with a monotone symplectic
form $\omega$. By Weinstein's neighborhood theorem one can fix two
symplectic embeddings $\phi_1$, $\phi_2$: $T_r^*S^2\rightarrow
S^2\times S^2$  for some small $r>0$. For each $i$, consider the
geodesic flow on  $S^2$ with the standard round metric. By
performing symplectic cut on $(S^2\times S^2, \omega)$
 along the boundary of the image of $\phi_i$, we obtain
 a pair of $S^2\times S^2$ for each $i$: one comes from $\phi_i(T_r^*S^2)$, equipped
 with the standard monotone symplectic form of size $r$;
 and the other one comes from the complement of $\phi_i(T_r^*S^2)$, equipped with symplectic form $\omega_i$
 and a symplectic $(-2)$-sphere $\Sigma_i$. Clearly, $[\omega_0]=[\omega_1]$.
  %where $M_i$ is identified with  $S^2\times S^2$ and
%got by a quotient of
%$S^1$-action on the boundary of the embedded $T_r^*S^2$.
%Under such an identification,
%the symplectic forms $\omega_i$ on $M_i$ are in the same class
%$PD(aA+bB)$, where $A,B$ are classes of the
%product factors and $a,b$ depends only on $r$.  Also, the
%$(-2)$-spheres $\Sigma_i$ are in the same homology class $A-B$.

%Now from the classification of symplectic forms on $S^2\times S^2$
%by \cite{LM1}\cite{LM2},
%There is a symplectomorphism
%$\Phi':M_1\rightarrow M_2$. Moreover,

It follows from the uniqueness of homologous symplectic structures
in \cite {LM2} and Proposition \ref{negative connectedness theorem},
%$(-2)$-spheres in $M_i$ are isotopic to each other.
%Therefore, there is a Hamiltonian isomorphism on $M_2$ which
% flows $\Phi(\Sigma_1)$ to $\Sigma_2$.  In sum we obtain a
there is a symplectomorphism of  pairs:

$$\iota: ((S^2\times S^2, \omega_1), \Sigma_1)\rightarrow ((S^2\times S^2, \omega_2), \Sigma_2),$$

\nono where $\iota$ sends a neighborhood of $\Sigma_1$
symplectomorphically to one of $\Sigma_2$.  Via symplectic sum
(\cite{Gompfpaper}), which is the exact inverse of symplectic cut
(as pointed out by Gompf), $\iota$ leads to a symplectomorphism of
pairs $\Psi: ((S^2\times S^2, \omega), L_1)\rightarrow ((S^2\times
S^2, \omega), L_2)$.
%on $(M_i, \Sigma_i)$ with a copy of $(S^2\times
%S^2, \text{diagonal})$.
\begin{comment} Gompf pointed out to us that symplectic sum
can be achieved without perturbation as an inverse of symplectic
cut as follows, which seems to be well-known.  Let $(M,\Sigma)$ and
$(N,\Sigma')$ be two symplectic pairs, where $\Sigma$ and $\Sigma'$
are symplectomorphic whose normal bundles have opposite Euler
classes.
 Let $P$ be the (real) projectivization of one of the normal bundles, then $P\times \RR$ has a canonical symplectic form with Hamiltonian $S^1$-action
 rotation each fiber.  The symplectic cut on $P\times \RR$ at $0$ gives two $\RR^2$-bundles on $\Sigma$ and $\Sigma'$ with the same Euler number as they are
 embedded in $M$ and $N$, respectively.  Therefore, the complement of $P\times\{0\}$ can be locally identified with the two normal bundles of $\Sigma$ and
$\Sigma'$ removing the zero section.
With this method, $\iota$ is
easily seen to be glued with the identity isomorphism of $S^2\times
S^2$, which leads to a symplectomorphism of pairs $\Psi: (S^2\times
S^2, L_1)\rightarrow (S^2\times S^2, L_2)$.
\end{comment}
\pftwo

%\subsection{Uniqueness up to isotopy} \label{section: Lagrangian
%uniqueness}

%\subsubsection{Hind's results on $S^2\times S^2$ and $T^*S^2$ and the symplectic mapping class group}

%{Conenctedness of symplectomorphism groups and
%isotopy}\label{cpt vs non-cpt}

\subsubsection {$T^*S^2$ and the symplectic mapping class group}

Further exploring the symplectic cut approach in  \ref{section:alternative approach-cut}, we
obtain an alternative proof of Hind's Lagrangian sphere
uniqueness in $T^*S^2$ below
via Seidel's description of the compactly supported symplectomorphism group of $T^*S^2$.
%discuss  the interactions between the symplectic
%mapping class group and Lagrangian isotopy problem for $T^*S^2$.

%we also :

\thmone[Hind, \cite{Hind2}]\label{Hind's theorem} Lagrangian
spheres in $(T^*S^2,\omega_{std})$ are unique up to Hamiltonian  isotopy. \thmtwo

\pfone[Proof:]
Via the negative Liouville flow and scaling we can isotope any Lagrangian in  $(T^*S^2,\omega_{std})$ into one in $(T_1^*S^2,\omega_{std})$.
Further, via the  identification
$(T_1^*S^2, \omega_{std})=(S^2\times S^2, \omega_0)\backslash \Delta$, where $\omega_0$ is a monotone form and
$\Delta$ is the diagonal of $S^2\times S^2$, it suffices to show the the uniqueness of
Lagrangian spheres in $(S^2\times S^2, \omega_0)\backslash \Delta$.

Given two Lagrangian spheres  $L_1$, $L_2\in
(S^2\times S^2, \omega_0)\backslash \Delta$, we first claim that there is
$\phi\in Symp_c(T_1^*S^2,\omega_{std})$ such that $\phi(L_1)=L_2$,
where $Symp_c$ denotes the compactly supported symplectomorphism group.

Without loss of generality we assume $L_2=\bar{\Delta}$, which is
the antidiagonal, corresponding in turn to the zero section of
$T^*S^2$.
  By Proposition \ref{Hind's}, there is $\Psi\in Symp(S^2\times S^2,
\omega_0)$, such that $\Psi(L_1)=L_2$.  $\Psi$ may not fix $\Delta$,
but notice that $\Psi(\Delta)\cap\bar{\Delta}(=L_2)=\emptyset$.
Since the complement of $\bar{\Delta}$  is canonically identified
with a symplectic disk bundle over the diagonal, by  \cite{HI} there
is a symplectic isotopy $\tilde{\Phi}_t:S^2\rightarrow (S^2\times
S^2, \omega_0)$ fixing $\bar{\Delta}$ and connecting the two
symplectic spheres $\Psi(\Delta)$ and $\Delta$.
% i.e $\tilde{\Phi}_0(S^2)=\Psi(\Delta)$, $\tilde{\Phi}_1(S^2)=\Delta$
%  so that
In particular, $\tilde{\Phi}_t\circ\Psi(\Delta)$ is disjoint from
$\bar{\Delta}$ for each $t$.

Now we extend $\tilde{\Phi}_t$ to a symplectic isotopy
 of a neighborhood $U$ of $\Psi(\Delta)$ which we
still denote as $\tilde{\Phi}_t$ (Ex. 3.40 in \cite{MSI}), and
require that $\tilde{\Phi}_t(U)$ be still disjoint from
$\bar{\Delta}$ for all $t$. We then trivially extend
$\tilde{\Psi_t}$ to $\tilde{\phi}_t$, a symplectic isotopy on a
neighborhood $U'$ of $\Psi(\Delta)\cup \bar{\Delta}$,
 which restricts to $\tilde{\Phi}_t$ on $U$ and to the  identity near $\bar{\Delta}$.
 Since $H^1(U';\RR)=0$, $H^2(S^2\times S^2, U';\RR)$ injects into $H^2(S^2\times S^2;\RR)$.
 By the argument proving Banyaga's isotopy extension theorem (see for example \cite{MSI},
Theorem 3.19), $\tilde{\phi_t}$ extends to a global symplectic
isotopy $\phi_t$ of $(S^2\times S^2,\omega_0)$, where $\phi_0=id$,
$\phi_1(L_1)=L_2$, and $\phi_1|_{\Delta}=id$.

Consider $\phi'=\phi_1\circ\Psi\in Symp(S^2\times S^2, \omega_0)$.
 Since $\phi'$ is the identity on $\Delta$, it induces a compactly supported
symplectomorphism $\phi$ of $(T_1^*S^2, \omega_{std})$ up to
isotopy, mapping $L_1$ to the zero section $L_2$.

From Seidel's description of
$Symp_c(T_1^*S^2,\omega_{std})$ in \cite{Seidel}, $\phi=\tau^n\circ
\eta_1$, where $\tau$ is the Lagrangian Dehn twist along the zero section $L_2$,
and $\eta_t$, $t\in[0,1]$ with $\eta_0=id$ is a compactly supported
symplectic isotopy. Now it is clear that $\tau^n\circ \eta_t(L_2)$
is a path connecting $L_1$ to the zero section since $\tau$ fixes
the zero section. \pftwo

\subsection{Proof of Theorem \ref{Lagrangian isotopy uniqueness}}

%  For simplicity we only show the theorem for class
%$A-B$, where $A$ and $B$ are the class of two product factors in
%$S^2\times S^2$.  From the point of view that $V_k$ is a
%$(k+1)$-ball blow-up of $\CC P^2$, this is the binary case. The
%ternary case should be clear from the proof and the ternary
%case in \cite{Evans}.

\nono
For $k\geq 0$ we will denote by $V_k$ the manifold $(S^2\times S^2)\# k\overline{\CC P}^2$. When $k\geq 1$,
$V_k=\CC P^2\# (k+1)\overline{\CC P}^2$.
 Due to Theorem \ref{Hind's theorem S2 times S2} and the fact that $\CC P^2\# \overline{\CC P}^2$ has no
 spheres with self-intersection $-2$, we only need to prove Theorem \ref{Lagrangian isotopy uniqueness} for  $V_k$ with $k=1,3$, and $k=2$ but $[L]$  not characteristic.
By Proposition \ref{K-Lagrangian sphere class classification}, we may further assume that $[L]$ is the binary class $E_1-E_2$.

Throughout this subsection, $J_0$ denotes the
complex structure obtained from a generic $k$-point complex blow-up of  $\CC P^1\times \CC P^1$.
 Without loss of generality, we may assume $\omega$ is a K\"ahler form
compatible with  $J_0$.  This follows from Proposition 4.8 in \cite{Li2} that
the  symplectic cone is the same as the $J_0$-compatible cone in
$H^2(V_k,\RR)$ when $k\leq 8$, as well as the uniqueness of homologous symplectic forms in
\cite{MIsotopy}.

To prove Theorem \ref{Lagrangian isotopy uniqueness}, we apply Theorem \ref{main} and follow the approach in
\cite{Evans} where  the   monotone case is settled.
  For some of the details one is referred to Section 9
of \cite{Evans} and 4.2 of \cite{EvansT}.

% The first step is, given  isotoping a stable symplectic sphere configuration (Definition \ref{configuration})  off a given Lagrangian sphere, and apply Theorem \ref{Hind's theorem} in the complement of the
%configuration.
%Each configuration is described by the set $D_L$ of the homology classes
%of the components.

For  the binary class $E_1-E_2$, the following stable symplectic sphere configuration \textit{type} (Definition \ref{configuration}) $D_{E_1-E_2}$  is introduced in \cite{Evans}:

\begin{itemize}
 \item  $\{H-E_1-E_2, H\}$ when $k=1$,
 \item  $\{H-E_1-E_2, H-E_3, E_3\}$ when $k=2$,
 \item  $\{H-E_1-E_2, H-E_3-E_4, E_3,E_4\}$ when $k=3$.
\end{itemize}

Since $(V_k, J_0)$ is a generic blow up, it is clear that there is a $J_0-$holomorphic $D_{E_1-E_2}$ configuration $C_0$.

\begin{lemma}\label{off}
Suppose $L$ is a Lagrangian sphere in $(V_k, \omega)$ with $k\leq 3$ and  $[L]=E_1-E_2$. Then $L$ can be Hamiltonian isotoped off $C_0$.
\end{lemma}

\pfone
From Corollary \ref{lemma:configuration existence}, in the complement of the given Lagrangian sphere $L$, we
can find a $D_{E_1-E_2}$-configuration $C$.
%consisting of $J$-holomorphic spheres with
%some almost complex structure $J$ compatible by $\omega$.

By Corollary \ref{symplectic connectedness'}, $C_0$ and $C$ are symplectically isotopic.
Following the proof of Theorem 9 in \cite{EvansT}, with a small
perturbation along the isotopy, we may assume the  symplectic spheres in the
configuration intersect $\omega$-orthogonally during the isotopy. Thus, by the symplectic neighborhood theorem,  we can extend this isotopy to
a neighborhood of the configuration.
From the fact that $C$ and $C_0$ have trivial
$H^1$, as in the proof of Theorem \ref{Hind's theorem},
we  obtain an ambient Hamiltonian isotopy $\Psi_t$ taking $C$ to $C_0$. In particular,
$L$ is Hamiltonian isotopic to $\Psi_1(L)$ which is disjoint from $C_0$.
% For example,  in the binary case when $k=2$, the standard configuration refers to
%a configuration of $J_0$-holomorphic curves in class $D$ $\{(H-E_1-E_2)(J_0), (H-E_3)(J_0), E_3(J_0)\}$.
%Therefore, one could assume the two homologous Lagrangians
%$L_1$ and $L_2$ both lie in the complement of $C_0$.

\pftwo

\begin{prop}\label{complement-unique-rational} Suppose there  is a Lagrangian sphere $L$ in $(V_k, \omega)$ with $k\leq 3$ and  $[L]=E_1-E_2$.
 When $[\omega]$ is a rational, the complement of $C_0$ contains a unique Lagrangian sphere up to Lagrangian isotopy.

\end{prop}
\pfone

%We begin with the case when $[\omega]$ is  rational.
%\\\\

%\nono{\bf {1. When $[\omega]$ is a rational}:}
%\\\\

By Lemma \ref{off} we can assume that the Lagrangian sphere $L$ is in the complement of $C_0$, so the complement of $C_0$ contains at least one  Lagrangian sphere.

We will discuss   the case $k=3$. The cases $k=1,2$ are similar.
Up to scaling,  we can write  $PD([\omega])=aH-E_1-E_2-b_3E_3-b_4E_4$ since
$\omega([L])=0$. Further,
%$b>0$ since $\omega(E_3)>0$,  $a>2$ since $\omega(H-E_1-E_2)>0$, and
$ a>1+b_i$ since $\omega(H-E_1-E_i)>0$ for $i=3,4$.
Rewrite $$PD([\omega])=(H-E_1-E_2)+(a-1)(H-E_3-E_4)+(a-1-b_3)E_3+(a-1-b_4)E_4.$$
Notice that   $a,b_i\in\QQ^+$ since $[\omega]$ is assumed to rational.
Since all coefficients are rational and positive,
there is  a large
integer $l$, such that $PD([l\omega])$ is represented as an
positive integral combination of $\{H-E_1-E_2, H-E_3-E_4, E_3, E_4\}$,
say, with coefficients $u, v, w, z\in\mathbb{Z}^+$.

If $C_0=C_{H-E_1-E_2}\cup C_{H-E_3-E_4}\cup C_{E_3}\cup C_{E_4}$, consider  the divisor
$F=uC_{H-E_1-E_2}+vC_{H-E_3-E_4}+wC_{E_3}+zC_{E_4}$. There is a holomorphic line
bundle $\LL$ with a  holomorphic section
$s$ whose zero divisor is exactly $F$.
Take an hermitian metric and a
compatible connection on $\LL$ such that the curvature form is just
$l\omega$.
$\phi=-log|s|^2$ defines a  plurisubharmonic function with
%$\partial\bar{\partial}\phi=l\omega$ on the complement $U_0$ of the
$-d(d\phi\circ J_0)=l\omega$ on the complement $U_0$ of the
$C_0$.

Notice that $U_0$ is   the same as the complement $U$ in Proposition 4.2.1 in
\cite{EvansT}, which is shown to be biholomorphic to the affine quadric there.
The rest of the argument is exactly as in the proof of Proposition 4.2.1 in \cite{EvansT},
%\cite{Evans},
reducing to  Theorem \ref{Hind's theorem}, the uniqueness in $(T^*S^2, \omega_{std})$.

Consider the  finite type Stein structure $(J_0, \phi/l)$ on  $U_0$.
Define $h:\RR\to \RR$ to be the function $h(x)=e^x-1$ and
$\phi_h=h\circ \phi$. By Lemma 3.1 in \cite{BCi} and  Lemma 6 in \cite{SI}, $(U_0, J_0, \phi_h)$ is a complete Stein manifold of finite type with K\"ahler form $\omega_h=-d(d\phi_h\circ J_0)$.
Suppose a sublevel set $Y=\phi^{-1}[0, k]$ contains all the critical points of $\phi$. View $(Y, \omega)$ as a Liouville  domain, and let   $(\hat Y, \hat \omega)$ be its symplectic completion.
By Lemma 2.1.5 in \cite{EvansT},
$(U_0, \omega_h)$ is symplectomorphic to  $(\hat Y, \hat \omega)$.

Since the  affine quadric $Q$ has a complete finite type Stein structure inherited from $\CC^3$,
it follows from Lemma  2.1.6 in
\cite{EvansT}  that   $(U_0, \omega_h)$ is symplectomorphic to
$(Q, \omega_{can})$.  Combining all the symplectomorphisms, we find that the Liouville manifold
$(\hat Y, \hat \omega)$  is symplectomorphic to $(T^*S^2, \omega_{std})$.

Given  any two
Lagrangian spheres $L_0$, $L_1$ in the complement of $C_0$, they lie in a sublevel set $Y$ of $\phi$ containing all the critical points.
We obtain an isotopy $L_t$ in $(\hat Y, \hat \omega)$  by Hind's Theorem \ref{Hind's theorem}.
Contract the  isotopy $L_t$ into the sublevel set $Y$ using the negative  Liouville flow on $(\hat Y, \hat \omega)$.
The endpoints of the contracted
isotopy are also connected in $Y$ to $L_0$ and $L_1$ respectively by the positive  Liouville flow. Therefore, one gets the desired
Hamiltonian isotopy between $L_0$ and $L_1$ in $Y\subset U_0$.

\pftwo

\pfone[Proof of Theorem \ref{Lagrangian isotopy uniqueness}:] As mentioned in the beginning of this subsection, we could assume that
$M=V_k$ with $k=1, 2, 3$, $\omega$ is a K\"ahler form compatible with $J_0$, and $\xi=E_1-E_2$.

Suppose $L_0$ and $L_1$ are two Lagrangian spheres in the class $\xi$.
By  Lemma \ref{off} they are Hamiltonian isotopic respectively to two Lagrangian spheres, still denoted by $L_0$ and $L_1$,  in the complement $U_0$ of $C_0$.
We will show that $L_0$ and $L_1$ are Lagrangian isotopic in $U_0$, and hence in $(V_k, \omega)$. As argued in Theorem \ref{Hind's theorem}, this implies that $L_0$ and $L_1$ are Hamiltonian isotopic.

Again we will discuss the case $k=3$.
By rescaling the symplectic form, we could still assume the
$\omega$-area of $E_1$ and $E_2$ is rational.
View $(V_3, \omega)$  as a three point blow-up of a monotone $(S^2\times S^2, \tau)$, then as the three disjoint components of $C_0$, $C_{H-E_1-E_2}, C_{E_3}, C_{E_4}$ are all exceptional, corresponding to three ball embeddings $h_{12}, e_3, e_4$
 in $(S^2\times S^2, \tau)$. Let $\tilde L_0$ and $\tilde L_1$ be the corresponding Lagrangians in $(S^2\times S^2, \tau)$.

%In this case, we take the point of view that the complement of the
%configuration is symplectomorphic
%to the complement of $k$ balls embedded into $S^2\times S^2$ which are disjoint from $L_1$ and $L_2$.

%(continuity of packing)the symplectic description of the blow-up procedure (see \cite{MSI},
%Section 7.1)

Via  the correspondence of ball-embeddings and symplectic forms in the blown-up manifolds, one may deform
$\omega$ to $\omega'$ near $C_{H-E_1-E_2}, C_{E_3}, C_{E_4}$
such  that their $\omega'$-areas become rational.  In fact, from the continuity of ball embeddings, such a deformation can be chosen
to correspond to a slightly larger ball-embeddings $h'_{12}$, $e'_3$ and $e'_4$ in $(S^2\times S^2, \tau)$. Further, we may assume that the larger embedded balls are
still disjoint from $\tilde L_0$ and $\tilde L_1$.  And when such a perturbation is chosen small enough, $J_0$ is still
tamed by $\omega'$ so that the  configuration $C_0$ is still symplectic with respect to $\omega'$.

Notice that  $L_0$ and $L_1$ remain Lagrangian in $(V_k,\omega')$.
Notice also that  $[\omega']$ is rational,  so we have   a Lagrangian isotopy
 between $L_0$ and $L_1$ in $(V_3,\omega')$ by Proposition \ref{complement-unique-rational}.
It is important to observe that  such an isotopy  can be chosen to lie  inside
the complement of the $\omega'-$symplectic configuration $C_0$.

In particular, the isotopy does not intersect the spheres $C_{H-E_1-E_2}, C_{E_3}, C_{E_4}$. In turn it
gives rise to an isotopy between $\tilde L_0$ and $\tilde L_1$ in the complement of the images of $h'_{12}$, $e'_3$ and $e'_4$.
Since $h'_{12}$, $e'_3$ and $e'_4$ are extensions of $h_{12}$, $e_3$ and $e_4$, the isotopy between $\tilde L_0$ and $\tilde L_1$
lie in the complement of the images of $h_{12}$, $e_3$ and $e_4$.
Therefore it gives rise to  an isotopy between $L_0$ and $L_1$
in the complement of the spheres $C_{H-E_1-E_2}, C_{E_3}, C_{E_4}$ in $(V_k,\omega)$.

\pftwo

%For the ternary case when $k=3$, $M$ is a $4$-point blow-up of $\CC
%P^2$.  One could reduce the problem to the binary case by a smooth
%Dehn twist as in the proof of Theorem \ref{Lagrangian sphere class classification}.

\subsection{Smooth isotopy}

\pfone[Proof of Theorem \ref{smooth isotopy uniqueness}:]
By Proposition \ref{K-Lagrangian sphere class classification}, we again assume that
we are in the binary case $E_1-E_2$. Given two Lagrangian spheres $L_i$,
 following  \cite{EvansT},
 consider the classes $E_j$, $j\geq 3$.
 From Theorem \ref{main}, for each $i$, we can
 find a set of disjoint symplectic spheres in $E_j$, which are also disjoint from $L_i$.
 Applying Proposition \ref{symplectic connectedness'} to these two stable spherical symplectic configurations as above,
 we can assume that $L_i$ are both disjoint from a set of disjoint symplectic spheres $S_i$ in $E_j, j\geq 3$.

Blow down $S_i$ we obtain $(\CC P^2\#2\overline{\CC P^2}, \omega') $ with balls $B_j$ disjoint from $L_i$.
Let $L_t$ be a  Lagrangian isotopy between $L_i$ in  $(\CC P^2\#2\overline{\CC P^2}, \tilde \omega) $ from Theorem \ref{Lagrangian isotopy uniqueness}.
Viewed as a smooth isotopy, we can assume that $L_t$ is transversal to the centers $x_j$  of $B_j$, thus avoiding $x_j$.
Let $B'_j\subset B_j$ be a smaller ball not intersecting $L_t$. Let $\phi$ be a diffeomorphism from  $U'$, the  complement of $\cup B_j'$ to $U$,  the complement of $\cup B_j$,
which is identity near $L_i$.
Then $\phi(L_t)$ is a smooth isotopy between $L_i$ in $U$. Blowing up at $x_j$ by cutting $B_j$, we get back to $(M, \omega)$ and   a smooth isotopy between $L_i$ therein.

\pftwo

\subsection{Some remarks on uniqueness}

We end the paper with some discussions about uniqueness.

%\begin{itemize}
%\item

%\item(The case of Lagrangian $\RR P^2$)
\begin{comment}
\subsubsection{The characteristic and ternary class $H-E_1-E_2-E_3$}

For the ternary case when $k=2$, $[L_i]=H-E_1-E_2-E_3$, it seems possible
 to prove uniqueness when $[\omega]$ is rational.

In this case, $D_L$ is given by $\{H-E_1, H-E_2, H-E_3\}$,  where the three spheres intersects at a  common point.
 We cannot apply Sikorav's automatic transversality for nodal submanifolds, but could apply \cite{Barraud2}.

The irrational case is more challenging.
%\begin{itemize}
% \item $\{H-E_1, H-E_2, H-E_3\}$ when $k=2$,
%\item $\{H-E_1-E_4, H-E_2, H-E_3\}$ when $k=3$.
%\end{itemize}

One possible approach is to
artificially blow up a small ball disjoint from $L_i$'s to obtain an
exceptional class $E_4$.  One then reduce the problem with more
caution.  The smooth Dehn twist $\phi$ along a sphere in class
$H-E_1-E_2-E_4$ reduces the problem to $[L'_i]=E_4-E_3$ in
$(M,(\phi^{-1})^*(\omega))$. From the choice of $D$, one then
obtains a Lagrangian isotopy between $L'_i$ lying in the complement
of an exceptional sphere of class $H-E_1-E_2$.  Now apply
$\phi^{-1}$ to $M$, one obtains a Lagrangian isotopy between $L_1$
and $L_2$ disjoint from an exceptional sphere in class $E_4$.
Blowing down this sphere thus concludes our proof.
\end{comment}

\subsubsection{Lagrangian $\RR P^2$}\label{rp2}

The argument  in  \ref{section:alternative approach-cut}, with $(-2)$-spheres replaced by $(-4)$-spheres,
can be used to  prove that any two Lagrangian $\RR P^2$ in $(\CC P^2, \omega_{std})$ are symplectomorphic.  From Gromov's connectedness of $Symp(\CC P^2,
\omega_{std})$ in \cite{Gromov}, we then obtain a new proof of  the following result of Hind (\cite{Hind2}).

\thmone[Hind] Any two Lagrangian $\RR P^2$ in $\CC P^2$ are Hamiltonian
isotopic to each other. \thmtwo

\begin{comment}
 Notice that the symplectic mapping
class group of $\CC P^2\#k\overline{\CC P^2}$, $k\leq 4,$
 with monotone symplectic
forms are known due to the work of Evans \cite{EvansS} and
Pinsonnault independently. Once the Abreu-McDuff type theorem is
checked, along this line one can show that Lagrangian $\RR P^2$ in
such manifolds are Hamiltonian isotopic.

\end{comment}

%\end{itemize}

\subsubsection{Uniqueness up to symplectomorphisms}\label{us}

%\conjone\label{k-blowups} For any two homologous Lagrangian spheres
%$L_1$ and $L_2$ in a symplectic rational manifold $(M,\omega)$, there
%exists  $\phi\in Symp_h(M,\omega)$ such that $\phi(L_1)=L_2$.
%\conjtwo

% Combining an idea proposed by  Hind and techniques
%in \cite{MIsotopy},
   Conjecture
\ref{k-blowups} states that,
for any two homologous Lagrangian spheres
$L_1$ and $L_2$ in a symplectic rational manifold $(M,\omega)$, there
exists  $\phi\in Symp_h(M,\omega)$ such that $\phi(L_1)=L_2$.
It implies
the disconnectedness of homologically trivial symplectormophism
groups in the cases when there are non-isotopic Lagrangian spheres.

We outline a possible  approach to Conjecture \ref{k-blowups}.
One easily reduces the problem to the binary case as in the proof of
Theorem \ref{Lagrangian sphere class classification}.
Without
loss of generality, let $[L_i]=E_1-E_2$.

For each pair $(M,
L_i)$,  by Theorem \ref{main}, away from $L_i$, there is a set of disjoint $(-1)$ symplectic spheres $C^l_i, l=3, ..., k+1,$   with $[C_i^l]=E_l$ for $l=3, ..., k,$ and  $[C^{k+1}_i]=
H-E_1-E_2$.
Blowing down the $C_l$ yields two
$(k+1)$-tuples of $(\tilde{M_i}, \tilde L_i, B_i^l)$, $i=1,2$, $3\leq l\leq
k+1$.   Here $\tilde{M_i}$ is a symplectic $S^2\times S^2$, $\tilde L_i$ a
Lagrangian sphere, and $B_i^l$ a symplectic  ball
corresponding to  $C_i^l$.

By \cite{LM2}  there is a symplectomorphism
$\Psi:\tilde{M_1}\rightarrow\tilde{M_2}$.
From  Theorem \ref{Hind's}, there is a
symplectomorphism sending $\Psi(\tilde L_1)$ to $\tilde L_2$.  Composing these
two symplectomorphisms one obtains a symplectomorphism between the
pairs $(\tilde{M_i}, \tilde L_i)$, which we still denote as $\Psi$. The conjectured
connectedness of relative symplectic ball embedding in Remark
\ref{New Connectedness Theorem} implies that the $k-2$ balls $\Psi(B_1^l)$ can be
 further displaced by an $\tilde L_2$-preserving Hamiltonian isotopy to
the balls $B_2^l$.
This gives a symplectomorphism between the $(k+1)$-tuples
$(\tilde{M_i}, \tilde L_i, B_i^l)$, which in turn descends to a
symplectomorphism  between the pairs $(M, L_i)$.

\begin{comment}
For  the ternary case, for example, $H-E_1-E_2-E_3$,
one can argue word-by-word as above on the balls $B_i^l$, $4\leq
l\leq k$, and instead of Hind's uniqueness, we apply Theorem
\ref{Lagrangian isotopy uniqueness}.
%\pftwo
\end{comment}

\subsubsection{Lagrangian $T^2$}

\textsc{1. School of Mathematical Sciences, University of Minnesota,
Minneapolis, MN55455, U.S.A.}

\textsc{2. School of Mathematical Sciences, University of Minnesota,
Minneapolis, MN55455, U.S.A.}

\begin{thebibliography}{[1]}
\bibliographystyle{amsplain}
\setlength{\itemsep}{0ex} \addcontentsline{toc}{section}{Reference}

\bibitem {AC}M. Alberich-Carraminana,  \textit{Geometry of the plane Cremona maps.}
 LNM 1769. Springer-Verlag, Berlin,2002.

\bibitem {AM} M. Abreu,  D. McDuff, \textit{Topology of symplectomorphism groups of rational ruled surfaces}.  J. Amer. Math. Soc.  13  (2000),  no. 4, 971--1009.

\bibitem {Audin} M. Audin, \textit{Lagrangian skeletons, periodic geodesic flows and symplectic cuttings}.  Manuscripta Math.  124  (2007),  no. 4, 533--550.

\bibitem {Barraud} J.-F. Barraud, \textit{  Nodal symplectic spheres in $CP^2$ with positive self-intersection},  Internat. Math. Res. Notices  1999,  no. 9, 495--508.

\bibitem {Barraud2} J.-F. Barraud, \textit{ Courbes pseudo-holomorphes �quisinguli�res en dimension 4.  [Equisingular pseudoholomorphic curves in 4-dimensional almost complex manifolds]}
Bull. Soc. Math. France 128 (2000), no. 2, 179�206.





\bibitem {Biran} P. Biran,  \textit{Connectedness of spaces of symplectic embeddings}.  Internat. Math. Res. Notices  1996,  no. 10, 487--491.

\bibitem {packing} P.  Biran,  \textit{Symplectic packings in dimension $4$}, Geom. Func. Anal. 7 (1997), no. 3, 420-437.

\bibitem {BiranS} P. Biran, \textit{A stability property of symplectic packing}. Invent. Math. 136 (1999), no. 1. 123-155.

\bibitem{BCi} P. Biran, K. Cieliebak, \textit{Symplectic topology on subcritical manifolds}, Comm. Math. Helv. 76  (2001),  712-753.

\bibitem {BC} P. Biran,  O. Cornea,  \textit{Quantum Structures for Lagrangian Submanifolds},  arxiv:0708.4221.

\bibitem {B} F. Bourgeois,  \textit{A Morse-Bott approach to contact homology}, Ph.D. thesis, NYU.

\bibitem {compactness} F. Bourgeois, Y. Eliashberg,  H. Hofer,  K. Wysocki,  E. Zehnder,  \textit{Compactness results in symplectic field theory.} Geom. Topol. 7 (2003), 799--888.

%\bibitem [Co05]{Coffey}  Coffey, Joseph, \textit{Symplectomorphism groups and isotropic skeletons}, Geom. Topol. 9 (2005), 935--970

\bibitem {CoZ} C. Conley, E. Zehnder,  \textit{Morse type index theory for flows and periodic solutions for Hamiltonian equations}, Comm. Pure Appl. Math.37 (1984),
207-253.

\bibitem {symfield}  Y. Eliashberg, A. Givental, H. Hofer,  \textit{ Introduction to symplectic field theory}
GAFA 2000 (Tel Aviv, 1999). Geom. Funct. Anal. 2000, Special Volume,
Part II, 560--673.

\bibitem {Evans} J. Evans, \textit{Lagrangian spheres in Del Pezzo surfaces}, J. Topol. 3 (2010), no. 1, 181-227.

\bibitem {EvansS} J. Evans,  \textit{Symplectic mapping class groups of some Stein and rational surfaces}, J. Symplectic Geom. 9 (2011), no. 1, 45-82.

\bibitem {EvansT} J. Evans, \textit{Symplectic topology of some Stein and rational
surfaces}, Ph. D. thesis, University of Cambridge.

%\bibitem [F-O99]{Fukaya} K. Fukaya, K. Ono, \textit{Arnold conjecture and Gromov-Witten
%invariant} Topology 38 (1999), no. 5, 933--1048.

\bibitem {FS} R. Fintushel and R. Stern; \textit{Invariants of Lagrangian tori}, Geom. Topol. 8 (2004), 947-968.
\bibitem {FM1} R. Friedman,  J. Morgan,  \textit{On the diffeomorphism types of certain algebraic surfaces. I}. J. Differential Geom. 27 (1988), no. 2, 297-369.

\bibitem {Gao}  H. Gao,  \textit{Representing homology classes of 4---manifolds}, Topology and its Application, 52(2) (1993), pp. 109-120.

\bibitem {Gompfpaper} R. Gompf, \textit{A new construction of symplectic manifolds}. (English summary)
Ann. of Math. (2) 142 (1995), no. 3, 527--595.

\bibitem {Gromov} M. Gromov,  \textit{Pseudo holomorphic curves in symplectic manifolds},
 Invent. Math. 82 (1985), no. 2, 307--347.

\bibitem {Hind} R. Hind, \textit{Lagrangian spheres in $S^2\times S^2$}.  Geom. Funct. Anal.  14  (2004),  no. 2, 303--318.

\bibitem {Hind2} R. Hind,  \textit{Lagrangian isotopies in Stein
manifolds}, arxiv:0311093

\bibitem {HI} R. Hind,  A. Ivrii, \textit{ Ruled 4-manifolds and isotopies of symplectic
surfaces.}  Math. Z.  265  (2010),  no. 3, 639--652.

\bibitem {HLS97} H. Hofer, V. Lizan, and J.-C. Sikorav, \emph{On
genericity for holomorphic curves in four-dimensional almost-complex
manifolds}, J. Geom. Anal. 7 (1997), no. 1, 149--159




\bibitem {IS99} S. Ivashkovich and V. Shevchishin, \emph{Structure of
the moduli space in a neighborhood of a cusp-curve and meromorphic hulls},
Invent. Math. 136 (1999), no. 3, 571--602.

\bibitem {Ki} K. Kikuchi,  \textit{Positive 2-spheres in 4-manifolds of signature $(1,n)$(1,n)}, Pacific J. Math., 160 (1993), pp. 245-258.

\bibitem {Lalonde} F. Lalonde,  \textit{Isotopy of symplectic balls, Gromov's radius and the structure of ruled symplectic $4$-manifolds.}  Math. Ann.  300  (1994),  no. 2, 273--296

\bibitem {LM2} F. Lalonde, D. McDuff,  \textit{$J$-curves and the classification of rational and ruled symplectic $4$-manifolds.}
Contact and symplectic geometry,
%(Cambridge, 1994),
3--42, Publ.
Newton Inst., 8, Cambridge Univ. Press, Cambridge, 1996.

%\bibitem {LM1} Lalonde, F.; McDuff, D., \textit{The classification of ruled symplectic $4$-manifolds.}  Math. Res. Lett.  3  (1996),  no. 6, 769--778

\bibitem {Lerman} E. Lerman,  \textit{Symplectic cuts}.  Math. Res. Lett.  2  (1995),  no. 3, 247--258.

\bibitem {Kod}  T.-J. Li, \textit{The Kodaira dimension of symplectic 4-manifolds}.
Floer homology, gauge theory, and low-dimensional topology,  249--261, Clay Math. Proc., 5, AMS Providence, RI, 2006.

\bibitem {Li^2} B. Li,   T.-J. Li, \textit{Symplectic genus, minimal genus and diffeomorphisms}.
  Asian J. Math.  6  (2002),  no. 1, 123--144.

\bibitem {Li} T.-J. Li, \textit{ Existence of symplectic surfaces.}
Geometry and topology of manifolds,  203--217, Fields Inst. Commun.,
47, AMS Providence, RI, 2005.

\bibitem {Li2} T.-J. Li \textit{The space of symplectic structures on closed 4-manifolds},
  3rd ICCM  259--277, AMS/IP Stud. Adv. Math., 42,  AMS, Providence, RI, 2008.

\bibitem {LL}  T.-J. Li, A.  Liu,  \textit{Uniqueness of symplectic canonical class,
surface cone and symplectic cone of 4-manifolds with $B^+=1$.}  J.
Differential Geom.  58  (2001),  no. 2, 331--370.

\bibitem {LL2} T.-J. Li,  A. Liu,  \textit{The equivalence between ${\rm SW}$ and ${\rm Gr}$ in the case where
$b^+=1$}. Internat. Math. Res. Notices 1999, no. 7, 335--345.

\bibitem {LU} T.-J. Li,  M. Usher, \textit{Symplectic forms and surfaces of negative square}. J. Symplectic Geom. 4 (2006), no. 1, 71-91.

\bibitem {LW} T.-J. Li,  W. Wu,  \textit{in preparation}.

\bibitem {Mcduff} D. McDuff,  \textit{The structure of rational and ruled symplectic $4$-manifolds}.  J. Amer. Math. Soc.  3  (1990),  no. 3, 679--712.

\bibitem {MB2} D. McDuff,  \textit{Remarks on the uniqueness of symplectic blowing up} in Symplectic Geometry; ed.
by D. Salamon; London Math. Soc. Lecture Note Ser. 192; Cambridge
Univ. Press; Cambridge; 1993; 157-67.

\bibitem {MIsotopy} D. McDuff,  \textit{From symplectic deformation to isotopy}.  Topics in symplectic $4$-manifolds (Irvine, CA, 1996),  85--99,
 Int. Press Lect. Ser., I, Int. Press, Cambridge, MA, 1998.

\bibitem {MP} D. McDuff, L. Polterovich,  \textit{Symplectic packings and algebraic geometry},
 With an appendix by Yael Karshon.  Invent. Math.  115  (1994),
 %no. 3,
 405--434.

\bibitem {MSI} D. McDuff,  D. Salamon, \textit{Introduction to symplectic topology}, 2nd edition.
Oxford Math. Mono.
%The Clarendon Press,
Oxford
University Press, New York, 1998

\bibitem {MSJ} D. McDuff, D. Salamon,  \textit{$J$-holomorphic curves and symplectic topology},
 AMS Coll. Pub., 52. AMS, Providence, RI, 2004.

\bibitem {MSE} D. McDuff,  F. Schlenk,  \textit{The embedding capacity of 4-dimensional symplectic ellipsoids},
arxiv:0912.0532

\bibitem {Pinsonnault} M. Pinsonnault, \textit{Maximal compact tori in the Hamiltonian group of 4-dimensional symplectic manifolds}. J. Mod. Dyn. 2 (2008), no. 3, 431–455,

\bibitem {Shev} V. Shevchishin,  \textit{Secondary Stiefel-Whitney class and diffeomorphisms of rational and ruled symplectic 4-manifolds}, Preprint, 50p., arXiv:0904.0283v2.

\bibitem {SZ} D. Salamon, E. Zehnder,  \textit{Morse theory for periodic solutions of of Hamiltonian systems and the Maslov index}, Comm. Pure and Appl. Math.
45 (1992), 1303-1360.

\bibitem {Seidel} P. Seidel, \textit{Symplectic automorphisms of $T^*S^2$}, arxiv:9803084.

\bibitem {Seidel notes} P. Seidel,  \textit{Lectures on four-dimensional Dehn twists},
  Symplectic 4-manifolds and algebraic surfaces,  231--267, LNM 1938, Springer, Berlin, 2008.

\bibitem {SI} P. Seidel, I.  Smith,  \textit{The symplectic topology of Ramanujam's surface},
  Comment. Math. Helv.  80  (2005),  no. 4, 859--881.

\bibitem{Sik} J.-C.  Sikorav, \textit{The gluing construction for normally generic $J$-holomorphic curves}, Symplectic and contact topology: interactions and perspectives (Toronto, ON/Montreal, QC, 2001), 175--199, Fields Inst. Commun., 35, Amer. Math. Soc., Providence, RI, 2003.

\bibitem {T1} C. Taubes,  \textit{Counting pseudo-holomorphic submanifolds in dimension 4}; J. Differential
Geom. 44 (1996); 818-93.

\bibitem {T2} C. Taubes, \textit{${\rm GR}={\rm SW}$: counting curves and connections}.
  J. Differential Geom.  52  (1999),  no. 3, 453--609.

  \bibitem{T3} C. Taubes,  \textit{Tamed to compatible: symplectic forms via moduli space integration}, to appear in J. Symplectic Geom.

\bibitem {Vi} S. Vidussi, \textit{Lagrangian surfaces in a fixed homology class: Existence of knotted Lagrangan tori}, J. Differential Geom. 74 (2006), 507-522. 

\bibitem {Wel} J. Welschinger,  \textit{Effective classes and Lagrangian tori in symplectic four-manifolds}.  J. Symplectic Geom.  5  (2007),  no. 1, 9--18.

\bibitem {Wendl} C. Wendl, \textit{Automatic transversality and orbifolds of punctured holomorphic curves in dimension four}. Comment. Math. Helv. 85 (2010), no. 2, 347-407.

\end{thebibliography}
\end{document}